\documentclass[a4paper,twoside,11pt]{article}

\usepackage[abbrvbib, preprint,nohyperref]{jmlr2e}

\usepackage[dvipsinames]{xcolor}
\definecolor{midnightblue}{HTML}{0059b3}
\definecolor{chromered}{HTML}{f14233}

\usepackage[utf8]{inputenc}
\usepackage{amsmath,mathtools}
\usepackage{amsfonts,amssymb}
\usepackage{dsfont}
\usepackage{graphicx,xcolor}
\usepackage{mathrsfs,booktabs}
\usepackage{nicefrac,a4wide,comment}
\usepackage{enumitem}

 \RequirePackage[colorlinks,citecolor=midnightblue,urlcolor=midnightblue,breaklinks=true,
linkbordercolor=midnightblue,linkcolor=midnightblue]{hyperref}

\usepackage{cleveref}
\usepackage{autonum}

\usepackage{tocloft}

\newcommand\bs{\boldsymbol}

\newcommand\bfC{\mathbf C}
\newcommand\bfE{\mathbf E}
\newcommand\bfP{\mathbf P}

\newcommand\bfH{\mathbf H}

\newcommand\bfM{\mathbf M}
\newcommand\bL{\boldsymbol L}

\newcommand\bW{\boldsymbol W}
\newcommand\bV{\boldsymbol V}
\newcommand\bh{\boldsymbol h}

\newcommand\bv{\boldsymbol v}
\newcommand\bg{\boldsymbol g}
\newcommand\bxi{\boldsymbol\xi}

\newcommand\NN{\mathbb N}
\newcommand\RR{\mathbb R}

\newcommand\bfI{\mathbf I}
\newcommand{\textupsf}[1]{\textup{\textsf{#1}}}

\newcommand\btheta{\boldsymbol\theta}
\newcommand\bvartheta{\boldsymbol\vartheta}

\def\tilde{\widetilde}
\def\hat{\widehat}

\def\bfE{\mathbf E}
\def\DKL{D_{\rm KL}}
\def\dTV{d_{\rm TV}}

\newtheorem{thm}{Theorem}
\newtheorem{prop}{Proposition}
\newtheorem{condition}{Condition}
\newtheorem{lem}{Lemma}

\newenvironment{myproof}[1] {\paragraph{Proof of {#1}  }}{\hfill$\blacksquare$}

\crefname{cor}{Corollary}{Corollaries}
\crefname{condition}{Condition}{C}
\crefname{prop}{Proposition}{Proposition}
\crefname{thm}{Theorem}{Theorem}
\crefname{lem}{Lemma}{Lemma}

\ShortHeadings{Langevin algorithms for log-concave targets}{Dalalyan, Karagulyan, Riou-Durand}
\firstpageno{1}

\begin{document}

\author{\name Arnak S. Dalalyan \email
    arnak.dalalyan@ensae.fr \\
    \name Avetik Karagulyan \email avetik.karagulyan@ensae.fr \\
    \addr ENSAE, CREST, IP Paris\\
       5 av. Le Chatelier, 
       91120 Palaiseau, France
    \AND
    \name Lionel Riou-Durand \email Lionel.Riou-Durand@warwick.ac.uk \\
    \addr University of Warwick\\
       Coventry CV4 7AL, United Kingdom
       }

\title{Bounding the error of discretized Langevin algorithms for
non-strongly log-concave targets}
\editor{}

\maketitle

\begin{abstract}
    In this paper, we provide non-asymptotic upper bounds on 
    the error of sampling from a target density over 
    $\mathbb R^p$ using three 
    schemes of discretized Langevin diffusions. The first 
    scheme is the Langevin Monte Carlo (LMC) algorithm, the 
    Euler discretization of the Langevin diffusion.	The second 
    and the third schemes are, respectively, the kinetic 
	Langevin Monte Carlo (KLMC) for differentiable potentials 
	and the kinetic Langevin Monte Carlo for twice-differentiable potentials (KLMC2).	The main focus is on the target 
	densities that are smooth and log-concave on $\RR^p$, 
	but not necessarily strongly log-concave. Bounds on the
	computational complexity are obtained under two types of 
	smoothness assumption: the potential has a 
	Lipschitz-continuous gradient and the potential has a
	Lipschitz-continuous Hessian matrix. The error of 
	sampling is measured by Wasserstein-$q$ distances. We 
	advocate for the use of a new dimension-adapted scaling 
	in the definition of the computational complexity, when
	Wasserstein-$q$ distances are considered. The obtained 
	results show that the number of iterations to achieve a 
	scaled-error smaller than a prescribed value depends 
	only polynomially in the dimension.
\end{abstract}



\setcounter{tocdepth}{1}

{
\renewcommand{\baselinestretch}{0.3}\normalsize
\tableofcontents
\renewcommand{\baselinestretch}{1.0}\normalsize
}
\section{Introduction}


The two most popular techniques for defining estimators or predictors
in statistics and machine learning are the $M$ estimation, also known
as empirical risk minimization, and the Bayesian method (leading to
posterior mean, posterior median, etc.).  In practice, it is necessary
to devise a numerical method for computing an approximation of these
estimators. Optimization algorithms are used for approximating an
$M$-estimator, while Monte Carlo algorithms are employed for
approximating Bayesian estimators. In statistical learning theory,
over past decades, a concentrated effort was made for getting non
asymptotic guarantees on the error of an optimization algorithm.
For smooth optimization, sharp results were obtained in the case
of strongly convex and convex cases \citep{Bubeck15}, the case
of non-convex smooth optimization being much more delicate \citep{Jain}.
As for Monte Carlo algorithms, past three years or so witnessed
considerable progress on theory of sampling from strongly 
log-concave densities. Some results for (non strongly) concave 
densities were obtained as well. However, to the best of our 
knowledge, there is no paper providing a systematic account 
on the error bounds for sampling from (non strongly) 
log-concave densities over unbounded domains. The main 
goal of this paper is to fill this gap.

A good starting point for accomplishing the aforementioned 
task is perhaps a result from \citep{durmus2019analysis} for 
the sampling error measured by the Kullback-Leibler divergence. 
The result is established for the Langevin Monte Carlo (LMC)
algorithm, which is the ``sampling analogue'' of the gradient 
descent. Let $\pi:\RR^p\to [0,+\infty)$ be a probability density
function (with respect to Lebesgue's measure) given by
\begin{align}
    \pi(\btheta)=\frac{e^{-f(\btheta)}}{\int_{\mathbb{R}^p} 
    e^{-f(\bv)}d\bv}.
\end{align}
for a potential function $f$. The goal of sampling is to 
generate a random vector in $\RR^p$ having a distribution 
close to the target distribution defined by $\pi$. In the 
sequel, we will make repeated use of the moments $\mu_k(\pi) 
= \bfE_{\bvartheta\sim\pi}[\|\bvartheta\|_2^k]^{1/k}$, where 
$\|\bv\|_q = (\sum_j |v_j|^q)^{1/q}$ is the usual 
$\ell_q$-norm for any $q\ge 1$. When there is no risk of
confusion, we will write $\mu_k$ instead of $\mu_k(\pi)$. 

To define the LMC algorithm, we need a sequence of positive
parameters $\bh=\{h_k\}_{k\in\NN}$, referred to as the 
step-sizes and an initial point $\bvartheta_{0,\bh}\in\RR^p$ 
that may be deterministic or random.
The successive iterations of the LMC
algorithm are given by the update rule
\begin{align}\label{LMC}
    \bvartheta_{k+1,\bh} = \bvartheta_{k,\bh} - h_{k+1} \nabla f(\bvartheta_{k,\bh})+ \sqrt{2h_{k+1}}\;\bxi_{k+1};
    \qquad k=0,1,2,\ldots
\end{align}
where  $\bxi_1,\ldots,\bxi_{k},\ldots$ is a sequence of 
independent, and independent of $\bvartheta_{0,\bh}$, 
centered Gaussian vectors with identity covariance 
matrices. Let $\nu_K$ denote the distribution of the $K$-th 
iterate of the LMC algorithm, assuming that all the
step-sizes are equal ($h_k=h$ for every $k\in\NN$) and 
the initial point is $\bvartheta_{0,h}=\mathbf 0_p$. 
We will also define the distribution $\bar\nu_K = 
(\nicefrac1K)\sum_{k=1}^K \nu_k$, obtained by choosing 
uniformly at random one of the elements of the sequence
$\{\bvartheta_{1,\bh},\ldots, \bvartheta_{K,\bh}\}$. It 
is proved in \citep[Cor.~7]{durmus2019analysis} that if 
the gradient $\nabla f$ is Lipschitz continuous with the
Lipschitz constant $M$, then for every $K\in\NN$, the 
Kullback-Leibler divergence between $\bar\nu_K$ and 
$\pi$  satisfies
\begin{align}\label{bound:1}
    \DKL(\bar\nu_K\|\pi) \le \frac{\mu^2_2(\pi)}{2Kh} 
    + Mph,\qquad
    \DKL(\bar\nu_K^{\rm opt}\|\pi) \le \sqrt{\frac{2Mp}{K}}\, 
    \mu_2(\pi).
\end{align}
Note that the second inequality above is derived from the 
first one by using the  step-size $h_{\rm opt} =  
\mu_2(\pi)/\sqrt{2KMp}$ obtained by minimizing the right 
hand side of the first inequality. Therefore, if we assume 
that the second-order moment $\mu_2$ of $\pi$ satisfies 
the condition $M\mu_2^2\le \kappa p^\beta$, for some 
dimension-free positive constants $\beta$ and $\kappa$, 
we get
\begin{align}
    \DKL(\bar\nu_K^{\rm opt}\|\pi) \le \sqrt{\frac{2\kappa 
    p^{1+\beta}}{K}}.
\end{align}
A natural measure of complexity of the LMC with averaging is, 
for every $\varepsilon>0$, the number of gradient evaluations
that is sufficient for getting a sampling error bounded from
above by $\varepsilon$. From the last display, taking into
account the Pinsker inequality, $\dTV(\bar\nu_K, \pi)\le
\sqrt{\DKL(\bar\nu_K,\pi)/2}$, and the fact that each iterate 
of the LMC requires one evaluation of the gradient of $f$, we
obtain the following result. The number of gradient evaluations
$K_{\rm LMCa, TV}(p,\varepsilon)$ sufficient for the 
total-variation-error of the LMC with averaging (hereafter, LMCa)
to be smaller than $\varepsilon$ is
\begin{align}
    K_{\rm LMCa, TV}(p,\varepsilon) = \frac{\kappa p^{1+\beta}}{2\varepsilon^4}.
\end{align}
The main goal of the present work is to provide this type 
of bounds on the complexity of various versions of the 
Langevin algorithm under different measures of the quality 
of sampling. The most important feature that we wish to 
uncover is the explicit dependence of the complexity 
$K(\varepsilon)$ on the dimension $p$, the 
inverse-target-precision $1/\varepsilon$ and the condition 
number $\kappa$. We will focus only on those measures
of quality of sampling that can be directly used for 
evaluating the quality of approximating expectations. 

The main contributions of this paper are: 
\vspace{-10pt}
\begin{itemize}\itemsep=0pt
    \item Bounds on the sampling error measured in the 
    Wasserstein $W_q$-distance (for $q\in[1,2]$) in two 
    settings: (a) convex and gradient-Lipschitz potential
    and (b) convex,  gradient-Lipschitz and Hessian-Lipschitz
    potential. Tight bounds are established for the 
    strong-convexified versions of three algorithms: 
    Langevin Monte Carlo, kinetic Langevin Monte Carlo and
    kinetic Langevin Monte Carlo for twice-differentiable 
    potentials. 
    \item Bounds on the moments of log-concave distributions, 
    especially in the settings where the potential function
    $f$ is strongly convex inside a ball and (weakly) convex
    everywhere else, or strongly convex outside a ball and
    (weakly) convex everywhere else. These bounds are 
    low-degree polynomials in dimension and in the other 
    relevant parameters. 
    \item Upper bounds on the mixing time of the 
    aforementioned three versions of the Langevin algorithm, 
    obtained by combining the $W_q$-error bounds and moment
    evaluations. 
\end{itemize}

The remainder of the paper is structured as follows. We begin in \Cref{sec:2} by introducing the sampling methods based on the
Langevin diffusion and by formulating the main assumptions that
are used throughout this work. \Cref{sec:3} and \Cref{sec:4} 
contain, respectively, a discussion on the choice of the
complexity measure and an overview of our main contributions. 
Relation to prior work is discussed in \Cref{sec:5}. 
\Cref{sec:precision_lmc} and \Cref{sec:precision_klmc} contain
formal statements and proofs of the main error- and
complexity-bounds for the vanilla and the kinetic Langevin, 
respectively. Upper bounds for the moments of two classes
of log-concave distributions are presented in \Cref{sec:moments}. 
We conclude with a discussion in \Cref{sec:disc}.

\begin{figure}
    \centering
    \includegraphics[width = 0.7\textwidth]{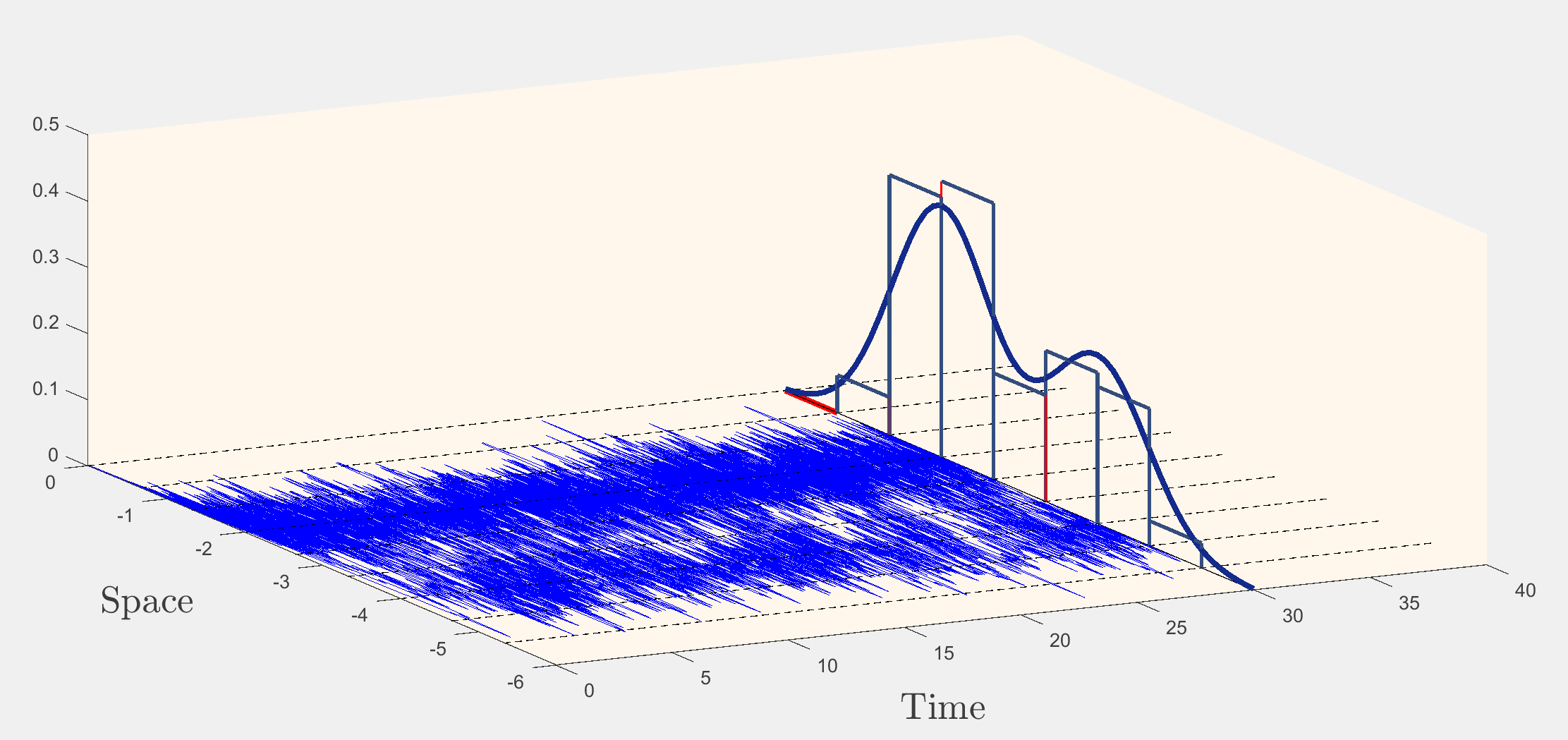}
    \caption{Illustration of Langevin dynamics. The blue lines represent different paths of a Langevin process. We see that the histogram of
    the state at time $t=30$ is close to the target density (the dark blue line).}
    \label{fig:1}
\end{figure}

\section{Further precisions on the analyzed methods}\label{sec:2}

Since our main motivation for considering the
sampling problem comes from applications in statistics and machine learning,
we will focus on the 
Monge-Kantorovich-Wasserstein distances $W_q$ 
defined by
\begin{align}
    W_q(\nu,\nu')  &= \inf\Big\{\bfE[\|\bvartheta-
    \bvartheta'\|_2^q]^{1/q}
    :\bvartheta\sim \nu \text{ and }
    \bvartheta'\sim \nu' \Big\},\qquad q\ge 1.
\end{align}
The infimum above is over all the couplings between $\nu$ 
and $\nu'$. In view of the H\"older inequality, the mapping
$q\mapsto W_q(\nu,\nu')$ is increasing for every pair
$(\nu,\nu')$. 

Our main contributions are upper bounds on quantities of
the form $W_q(\nu_K,\pi)$
where $\pi$ is a log-concave target distribution and $\nu_K$
is the distribution of the $K$th iterate of various discretization schemes of
Langevin diffusions. More precisely, we consider two types
of Langevin processes: the kinetic Langevin diffusion and
the vanilla Langevin diffusion. The latter is the
highly overdamped version of the former, see \citep{Nelson}.
The Langevin diffusion, having $\pi$ as invariant distribution,
is defined as a solution\footnote{Under the conditions imposed
on the function $f$ throughout this paper, namely the convexity
and the Lipschitzness of the gradient, all the considered
stochastic differential equations have unique strong
solutions. Furthermore, all conditions (see, for instance,
\citep{Pavliotis14}) ensuring that $\pi$ and $p^*$
are invariant densities of, respectively, processes
\eqref{langevin} and \eqref{kinetic} are fulfilled.}
to the stochastic differential equation
\begin{align}
    d\bL_t^{\textupsf{LD}} = -\nabla f(\bL_t^{\textupsf{LD}})\,dt 
    + \sqrt{2}\,d\bW_t,\qquad t\ge 0,
    \label{langevin}
\end{align}
where $\bW$ is a $p$-dimensional standard Brownian motion
independent of the initial value $\bL_0$. An illustration of this
process is given in \Cref{fig:1}. The LMC algorithm
presented in \eqref{LMC} is merely the Euler-Maruyama
discretization of the process $\bL$. The kinetic Langevin 
diffusion $\{\bL_t^{\textupsf{KLD}}:t\ge 0\}$, also known as 
the second-order Langevin process, is defined by
\begin{align}\label{kinetic}
    d
		\begin{bmatrix}
		\bV_t\\
		\bL_t^{\textupsf{KLD}}
		\end{bmatrix}
		&=
		\begin{bmatrix}
		-(\gamma \bV_t + \nabla f(\bL_t^{\textupsf{KLD}}))\\
		\bV_t
		\end{bmatrix}
		\,dt+\sqrt{2\gamma}
		\begin{bmatrix}
		\bfI_p\\
		\mathbf 0_{p\times p}
		\end{bmatrix}
		\,d\bW_t,\qquad t\ge 0,
\end{align}
where $\gamma>0$ is the friction coefficient. The process
$\bV_t$ is often called the velocity process since the second
row in \eqref{kinetic} implies that $\bV_t$ is the time
derivative of $\bL_t^{\textupsf{KLD}}$. The continuous-time 
Markov process $(\bL_t^{\textupsf{KLD}},\bV_t)$ is positive 
recurrent and has a unique 
invariant distribution, which has the following density 
with respect to the Lebesgue measure on $\RR^{2p}$:
\begin{align}
p_{*}(\btheta,\bv)\propto\exp\Big\{-f(\btheta) - \frac1{2}\|\bv\|_2^2\Big\}
,\qquad \btheta\in\RR^p,\ \bv\in\RR^p.
\end{align}
If $(\bL,\bV)$ is a pair of random vectors drawn from the
joint density $p_*$, then $\bL$ and $\bV$ are independent, 
$\bL$ is distributed according to the target $\pi$, whereas 
$\bV$ is a standard Gaussian vector. Therefore, at equilibrium, 
the random variable $\bL_t^{\textupsf{KLD}}$ has the target 
distribution $\pi$.

Time-discretized versions of Langevin diffusion
processes \eqref{langevin} and \eqref{kinetic}
are used for (approximately) sampling from $\pi$.
In order to guarantee that the discretization error is not
too large, as well as that the process $\{\bL_t\}$
converges fast enough to its invariant distribution,
we need to impose some assumptions on $f$. In the present
work, we will assume that either Conditions \ref{cond:1},
\ref{cond:2} or Conditions \ref{cond:1}, \ref{cond:2},
\ref{cond:3} presented below are satisfied.

\begin{condition}\label{cond:1}
    The function $f$ is continuously differentiable on 
		$\RR^p$ and its gradient $\nabla f$ is $M$-Lipschitz 
		for some $M>0$: $\|\nabla f(\btheta)-\nabla f(\btheta')\|_2\le
    M\|\btheta-\btheta'\|_2$ for all $\btheta,\btheta'\in\RR^p$.
\end{condition}

From now on, we will always assume that the Langevin (vanilla or 
kinetic) diffusion under consideration has the initial point
$\bL_0 = \mathbf 0$. Some of the conditions presented below 
implicitly require that this initialization is not too far
away from the ``center'' of the target distribution $\pi$. 
In many statistical problems where $\pi$ is the Bayesian posterior,
one can come close to these assumptions by shifting the distribution 
using a simple initial estimator. 

\begin{condition}\label{cond:2}
    The function $f$ is convex on $\RR^p$. Furthermore,
    for some positive constants $D$ and $\beta$, we have
    $\mu^2_2(\pi)= \bfE_{\bvartheta\sim \pi}[\|\bvartheta\|_2^2]
    \le D p^\beta$.
\end{condition}

Under \Cref{cond:2}, the centered second moment of $\pi$ scales 
polynomially with the dimension with power $\beta>0$, while the 
flatness of the distribution is controlled by the parameter $D>0$. 
Remarkably, \Cref{cond:2} implies that all the moments 
$\{\mu_q(\pi)\}_{q\ge1}$ scale polynomially with $p$, provided 
that $\bfE_{\vartheta\sim \pi}[\|\bvartheta\|_2^2]$ also does. 
This fact is a consequence of Borell's lemma \citep[Theorem~ 2.4.6 
]{Giannopoulos},  stating that for any $q\ge 1$, there is a 
numerical constant $B_q$ that depends only on $q$ such that
$\mu_q(\pi)\le B_q\mu_2(\pi)$. An attempt to provide optimized
constants in this inequality is stated in \Cref{lemma-ck}.

In the sequel, we show that the smoothness and the flatness of 
$\pi$ have a combined impact on the sampling error considered. 
It turns out that the important parameter with respect to the 
hardness of the sampling problem is the product
\begin{equation}
    \kappa:= MD.
\end{equation}
For $m$-strongly convex functions $f$, \Cref{cond:2} is satisfied 
with $D = 1/m$ and $\beta=1$, according to Brascamp-Lieb 
inequality \citep{BRASCAMP1976366}. In this case, the parameter
$\kappa=M/m$ is referred to as the condition number. We will show 
that \Cref{cond:2} is also satisfied for functions $f$ that 
are convex everywhere and strongly convex inside a ball, as 
well as for functions $f$ that are convex everywhere and 
strongly convex only outside a ball.

In the next assumption, we use notation $\|\bfM\|$
for the spectral norm (the largest singular value) of a
matrix $\bfM$.
\begin{condition}\label{cond:3}
    The function $f$ is twice differentiable in $\RR^p$ with
    a $M_2$-Lipschitz Hessian $\nabla^2 f$ for some $M_2>0$:
		$\|\nabla^2 f(\btheta)-\nabla^2 f(\btheta')\|\le
    M_2\|\btheta-\btheta'\|_2$ for all $\btheta,
		\btheta'\in\RR^p$.
\end{condition}
\Cref{cond:3} ensures further smoothness of the potential $f$. 
When it holds, the Lipschitz continuity of the Hessian and 
the flatness of $\pi$ also have a  combined impact on the 
sampling error. A second important parameter with respect 
to the hardness of the sampling problem in such a case is 
the product 
\begin{equation}
    \kappa_2:= M_2^{2/3}D.
\end{equation}

The case of an $m$-strongly convex function $f$ has been
studied in several recent papers. As a matter of fact,
global strong convexity implies exponentially fast mixing
of processes \eqref{langevin} and \eqref{kinetic}, with
dimension-free rates $e^{-mt}$ and $e^{-mt/(M+m)^{1/2}}$,
respectively. When only (weak) convexity is assumed, such
results do not hold in general. Therefore, the strategy we
adopt here consists in sampling from a distribution that is
provably close to the target, but has the advantage of being
strongly log-concave.

More precisely, for some small positive $\alpha$, the
surrogate potential is defined by $f_{\alpha}(\btheta)
:= f(\btheta)+\alpha\|\btheta\|_2^2/2$. Therefore,
the corresponding surrogate distribution has the
density
\begin{equation}\label{pi_alpha}
    \pi_{\alpha}(\btheta):=
		\frac{e^{-f_{\alpha}(\btheta)}}{\int_{\mathbb{R}^p}
		e^{-f_{\alpha}(\bv)}d\bv}.
\end{equation}
We stress the fact that the quadratic penalty $\alpha\|\btheta\|_2^2/2$ 
added to the potential $f$ is centered at the origin. This is closely
related to the fact that the diffusion is assumed to have the origin
as initial point, and also to the fact that the origin is assumed here
to be a good guess of the ``center'' of $\pi$. The parameter $\alpha$, 
together with the step-size $\bh$,
is considered as a tuning parameter of the algorithms to be
calibrated. Large values of $\alpha$ will result in
fast convergence to $\pi_\alpha$ but a poor approximation
of $\pi$ by $\pi_\alpha$. On the other hand,  smaller
values of $\alpha$ will lead to a small approximation error
but also slow convergence. The next result quantifies the
approximation of $\pi$ by $\pi_\alpha$, for different
distances.

\begin{prop}\label{prop:wass_vill}
For any $\alpha\ge0$ and $q\in [1,+\infty)$, there is a 
numerical constant $C_q$ depending only on $q$ such that
\begin{align}
    d_{\rm TV}(\pi,\pi_{\alpha})\le\alpha\mu^2_2(\pi),\qquad
    W_q^q(\pi,\pi_{\alpha})\le C_q\alpha\mu_2(\pi)^{q+2}.
\end{align}
Here the constant $C_q$ can be bounded for every $q$.
In particular, $C_1\le 11$ and $C_2\le 111$.
\end{prop}

This result allows us to control the bias induced by replacing
the target distribution by the surrogate one and paves the way
for choosing the ``optimal'' $\alpha$ by minimizing an upper 
bound on the sampling error. We draw the attention of the reader 
to the fact that, for $W_q$ distance, the dependence on $\alpha$ 
of the upper bound is $\alpha^{1/q}$, which slows down when $q$
increases (recall that $\alpha$ is a small parameter). This 
explains the deterioration with increasing $q$ of the complexity 
bounds presented in forthcoming sections. Let us define the constant
\begin{align}\label{Cq}
    C_q=\inf\{C:W_q^q(\pi,\pi_{\alpha})\le C\alpha\mu_2(\pi)^{q+2},
    \ \forall\,\pi \textup{ log-concave}\},   
\end{align}
which will repeatedly appear in the statements of the theorems.

\section{How to measure the complexity of a sampling scheme?}
\label{sec:3}

We have already introduced the notation $K_{\sf Alg,Crit} 
(p,\varepsilon)$, the number of iterations that guarantee that
algorithm {\sf Alg} has an error---measured by criterion 
{\sf Crit}---smaller than $\varepsilon$. If we choose a 
criterion, this quantity can be used to compare two methods, 
the iterates of which have comparable computational complexity. 
For example, LMC and KLMC being discretized versions of the 
Langevin process \eqref{langevin} and the kinetic Langevin 
process \eqref{kinetic}, respectively, are such that one 
iteration requires one evaluation of $\nabla f$ and generation 
of one realization of a Gaussian vector of dimension $p$ or 
$2p$. Thus, the iterations are of comparable computational 
complexity and, therefore, it is natural to prefer LMC if 
$2K_{\rm LMC,\sf Crit}(p,\varepsilon)\le K_{\rm KLMC,\sf Crit}
(p,\varepsilon)$ and to prefer KLMC if the opposite inequality 
is true. 

A delicate question that has not really been discussed 
in literature is a notion of complexity that allows to 
compare the quality of a given sampling method for two
different criteria. To be more precise,  assume that we are
interested in the LMC algorithm and wish to figure out whether
it is ``more difficult'' to perform approximate sampling for
the TV-distance  or for the Wasserstein distance. It is a
well-known fact that the TV-distance induces the uniform strong
convergence of measures whereas the Wasserstein distances
induce the weak convergence. Therefore, at least intuitively,
approximate sampling for the TV-distance should be harder than
approximate sampling for the  Wasserstein distance\footnote{We
underline here that the aforementioned hardness argument is
based only on the topological consideration, since it is not
possible, in general, to upper bound the Wasserstein distance
$W_q$, for $q\ge 1$ by the TV-distance or a function of it.}.
However, under \Cref{cond:1} and $m$-strong convexity
of $f$, the available results for the LMC provide the same
order of magnitude, $p/\varepsilon^2$, both for $K_{\rm
LMC,TV}$ \citep{Dalalyan14,Durmus2} and $K_{{\rm LMC}, W_2}$
\citep{Durmus2,colt_DalalyanT09}. The point we want to put
forward is that the origin of this discrepancy between
the intuitions and mathematical results is the inappropriate
scaling of the target accuracy in the definition of $K_{{\rm
LMC}, W_2}$. 

To further justify the importance of choosing the right scaling
of the target accuracy, let us make the following observation.
The total-variation distance, on the one hand, serves to 
approximate probabilities, which are adimensional and 
scale-free quantities belonging to the interval $[0,1]$. The Wasserstein distances, on the other hand, are useful for
approximating moments\footnote{Recall that by the triangle
inequality, one has $|\mu_q(\nu) - \mu_q(\pi)| \le W_q(\nu,\pi)
$.}, the latter depending  both on the dimension and on 
the scale. For this reason, we suggest the following definition 
of the analogue of $K$ in the case of Wasserstein distances:
\begin{align}\label{Keps}
K_{{\sf Alg},W_q}(p,\varepsilon) = 
\min\{k\in\NN : 
W_q(\nu_k^{\sf Alg},\pi)\le \varepsilon \mu_2(\pi),\ 
\forall\,\pi\in\mathscr P\},
\end{align}
where ${\sf Alg}$ is a Markov Chain Monte Carlo or another 
method of sampling, $k$ is generally the number of calls to 
the oracle and $\mathscr P$ is a class of target distributions.
Examples of oracle call are the evaluation of the gradient of
the potential at a given point or the computation of the 
product of the Hessian of $f$ at a given point and a given 
vector. Note also that $\mu_2(\pi)$ is the $W_2$ distance 
between the Dirac mass at the origin and the target distribution. 

Definition \eqref{Keps}, as opposed to those used in prior 
work, has the advantage of being scale invariant and 
reflecting the fact that we deal with objects that might 
be large if the dimension is large. Note that the idea of
scaling the error in order to make the complexity measure
scale-invariant has been recently used in \citep{lee2018,Baker2019}
as well. Indeed, in the context of $m$-strongly log-concave 
distributions, \cite{lee2018} propose to find the smallest
$k$ such that $W_2(\nu_k^{\sf Alg},\pi)\le\varepsilon/\sqrt{m}$. 
This is close to our proposal, since in the case of 
$m$-strongly log-concave distributions, it follows from 
the Brascamp-Lieb inequality that $\sup_\pi {\mu_2(\pi)} 
= \sqrt{p/m} $ (the sup is attained for Gaussian distributions).

\section{Overview of main contributions}\label{sec:4}

In this work, we analyze three methods, LMC, KLMC 
\citep{cheng2018underdamped} and KLMC2 \citep{dalalyan_riou_2018}, 
applied to the strong-convexified potential 
$f_\alpha(\btheta) = f(\btheta) + (\alpha/2)\|\btheta\|_2^2$ 
in order to cope with the lack of strong convexity. We briefly
recall these algorithms and present a summary of the main
contributions of this work. 

\subsection{Considered Markov chain Monte-Carlo methods}
We first recall the definition of the Langevin Monte Carlo
algorithms. For the LMC algorithm introduced in \eqref{LMC}, 
we will only use the constant step-size form, the update rule
of which is given by
\begin{align}\label{a-LMC}\tag{$\alpha$-LMC}
    \bvartheta_{k+1} = (1-\alpha h)\bvartheta_{k} - h \nabla f(\bvartheta_{k})+ \sqrt{2h}\;\bxi_{k+1};
    \qquad k=0,1,2,\ldots
\end{align}
where  $\bxi_1,\ldots,\bxi_{k},\ldots$ is a sequence of 
mutually independent,  independent of $\bvartheta_{0}$,
centered Gaussian vectors with covariance matrices equal to
identity. We will refer to this version of the LMC algorithm as
$\alpha$-LMC. 

We now recall the definition of the first and second-order
Kinetic Langevin Monte Carlo algorithms. We suppose that, for
some initial distribution $\nu_0$ chosen by the user, both KLMC 
and KLMC2 algorithms start from $(\bv_0,\bvartheta_0)\sim
\mathcal N(\mathbf 0_p,\bfI_p)\otimes \nu_0$. Before stating 
the update rules, we specify the structure of the random
perturbation generated at each step. In what follows,
$\{(\bxi_{k}^{(1)},\bxi_{k}^{(2)}, \bxi_{k}^{(3)},\bxi_{k}^{(4)}):
k\in\mathbb{N}\}$ will stand for a sequence of iid $4p$-dimensional
centered Gaussian vectors, independent of the initial condition
$(\bv_0,\bvartheta_0)$. 

To specify the covariance structure of these Gaussian variables, we define two sequences of functions $(\psi_k)$ and $(\varphi_k)$ as follows. For every $t>0$, let $\psi_0(t) = e^{-\gamma t}$, then for
every $k\in\mathbb{N}$, define $\psi_{k+1} (t)= \int_0^t \psi_k(s)\,ds$
and $\varphi_{k+1}(t) = \int_0^t e^{-\gamma (t-s)} \psi_k(s)\,ds$. Now,
let us denote by $\xi_{k,j}$ for the $j$-th component of the
vector $\bxi_{k}$ (a scalar), and assume that for any fixed $k$, the $4$-dimensional random vectors $\{(\xi_{k,j}^{(1)},\xi_{k,j}^{(2)},\xi_{k,j}^{(3)},\xi_{k,j}^{(4)}):1\le j \le p\}$ are iid with the covariance matrix
\begin{align}
\bfC_{h,\gamma} =
\int_0^h [\psi_0(t);\, \psi_1(t);\, \varphi_2(t);\, \varphi_3(t)]^\top
[\psi_0(t);\, \psi_1(t);\, \varphi_2(t);\, \varphi_3(t)]\,dt.
\end{align}
 The KLMC algorithm, introduced by \cite{cheng2018underdamped}, is a sampler derived
 from a suitable time-discretization of the kinetic diffusion. When
 applied to the strong-convexified potential $f_\alpha$, for a step-size
 $h>0$, its update rule reads as follows
\begin{align}\label{KLMC}\tag{$\alpha$-KLMC}
\begin{bmatrix}
	\bv_{k+1}\\[4pt]
	\bvartheta_{k+1}
\end{bmatrix}
&=
\begin{bmatrix}
	\psi_0(h)\bv_k-\psi_1(h)(\nabla f(\bvartheta_k)+\alpha\bvartheta_k)\\[4pt]
	\bvartheta_{k} + \psi_1(h)\bv_k -\psi_2(h)(\nabla f(\bvartheta_k)+\alpha\bvartheta_k)
\end{bmatrix}
+ \sqrt{2\gamma }
\begin{bmatrix}
	\bxi_{k+1}^{(1)}\\[4pt]
	\bxi_{k+1}^{(2)}
\end{bmatrix}.	
\end{align}
Roughly speaking, this formula is obtained from \eqref{kinetic}
by replacing the function $t\mapsto \nabla f(\bL_t)$ by a piecewise
constant approximation. Such an approximation is made possible by the
fact that $f$ is gradient-Lipschitz. 

It is natural to expect that further smoothness of $f$ may allow one
to improve upon the aforementioned piecewise constant approximation.
This is done by the KLMC2 algorithm, introduced by \cite{dalalyan_riou_2018}, which takes advantage of the existence and
smoothness of the Hessian of $f$ in order to use a local-linear
approximation. At any iteration $k\in\mathbb N$ with a current value
$\bvartheta_k$, define the gradient $\bg_{k,\alpha}=\nabla f(\bvartheta_k)+\alpha\bvartheta_k$ and the Hessian $\bfH_{k,\alpha} 
= \nabla^2 f(\bvartheta_k)+\alpha\bfI_p$. When applied to the modified strongly convex potential $f_\alpha$, for $h>0$, 
the update rule of the KLMC2 algorithm is
\begin{align}\label{KLMC2}\tag{$\alpha$-KLMC2}
\begin{bmatrix}
	\bv_{k+1}\\[4pt]
	\bvartheta_{k+1}
\end{bmatrix}
&=
\begin{bmatrix}
	\psi_0(h)\bv_k-\psi_1(h)\bg_{k,\alpha} -
	\varphi_2(h)\bfH_{k,\alpha}\bv_k\\[4pt]
	\bvartheta_{k} + \psi_1(h)\bv_k -\psi_2(h)\bg_{k,\alpha}
	-\varphi_3(h)\bfH_{k,\alpha}\bv_k
\end{bmatrix}
+ \sqrt{2\gamma }
\begin{bmatrix}
	\bxi_{k+1}^{(1)} -\bfH_{k,\alpha}\bxi_{k+1}^{(3)}\\[4pt]
	\bxi_{k+1}^{(2)} -\bfH_{k,\alpha}\bxi_{k+1}^{(4)}
\end{bmatrix}.
\end{align}
Notice that if we apply KLMC2 with $\bfH_{k,\alpha} =0$, we
recover the KLMC algorithm. These two algorithms, derived from
the kinetic Langevin diffusion, will be referred to as 
$\alpha$-KLMC and $\alpha$-KLMC2.

\subsection{Summary of the obtained complexity bounds}
\label{ssec:4.2}

Without going into details here, we mention in the tables below the
order of magnitude of the number of iterations required by different
algorithms for getting an error bounded by $\varepsilon$ for various
metrics. For improved legibility, we do not include logarithmic
factors and report the order of magnitude
of $K_{\square,\square}(p,\varepsilon)$ in the case when the parameter
$\beta$ in \Cref{cond:2} is fixed to a particular value. We present 
hereafter the case where $\beta=1$, which is of particular interest 
as discussed in \Cref{sec:moments}.
\begin{center}
  \begin{tabular}{c|c|c|c|c|c}
		\toprule
		$\beta=1$ & LMCa & \multicolumn{2}{c|}{\ref{a-LMC}} & 
		 \ref{KLMC} & \ref{KLMC2} \\
		\midrule
		Cond.& \ref{cond:1}-\ref{cond:2} & \ref{cond:1}-\ref{cond:2} & 
		\ref{cond:1}-\ref{cond:3} & \ref{cond:1}-\ref{cond:2} & 
		\ref{cond:1}-\ref{cond:3} \\
		\midrule
		$W_2$  & $-$ & $\kappa p^2/\varepsilon^6\ \hphantom{\color{chromered}\bs\square}$ &
		$(\kappa_2^{1.5}p^{0.5} + \kappa^{1.5}) {p^{2}}/{\varepsilon^{5}}$ & $\kappa^{1.5} p^{2}/\varepsilon^{5}$ 
		& $\kappa^{0.5} {\kappa_2}^{1.5}p^{2}/\varepsilon^{4}$ \\[4pt]
		$W_1$  & $-$ & $\kappa p^2/\varepsilon^4\ \hphantom{\color{chromered}\bs\square}$ &
		$ (\kappa_2^{1.5}p^{0.5} + \kappa^{1.5}){p^{2}}/{\varepsilon^{3}}$ & $\kappa^{1.5}p^{2}/\varepsilon^{3}$ 
		& $\kappa^{0.5} {\kappa_2}^{1.5}p^{2}/\varepsilon^{2}$ \\[4pt]
		$\dTV$ & $\kappa p^2/\varepsilon^4\ {\color{midnightblue}\bs\triangle}$ & $\kappa^2p^3/\varepsilon^4\
		{\color{chromered}\bs\square} $ & $-$ & $-$ & $-$ \\[4pt]
		\bottomrule
	\end{tabular}

\end{center}
The results indicated by ${\color{midnightblue}\bs\triangle}$ describe 
the behavior of the Langevin Monte Carlo with averaging established 
in \citep[Cor.~7]{durmus2019analysis}. 
To 
date, these results have the best known dependence (under 
conditions 1 and 2 only) on $p$. The results indicated by ${\color{chromered}\bs\square}$ summarize the
behavior of the Langevin Monte Carlo established in \citep{Dalalyan14}. All the
remaining cells of the table are filled in by the results obtained in the present
work. One can observe that the results for $W_1$
are strictly better than those for $W_2$. Similar
hierarchy was already reported in \citep[Remark~1.9]{majka2018}. 
It is also worth mentioning here, that using Metropolis-Hastings 
adjustment of the LMC (termed MALA), \cite{dwivedi2018log,ChenYuansi}
obtained\footnote{This rate is not explicitly present in the 
cited papers, but it can be derived from 
\citep[Theorem 5]{ChenYuansi} using the approach outlined 
in \citep[Section 3.3]{dwivedi2018log}.} the complexity
\begin{equation}
K_{\rm MALA,TV}(p,\varepsilon) = O\bigg(\frac{p^2\kappa^{3/2}}{\varepsilon^{3/2}}\,
\log\big(p\kappa/\varepsilon\big)\bigg).
\end{equation}
It is still an open question whether this type of result can be 
proved for Wasserstein distances. 

We would also like to comment on the relation between the 
third and the fourth columns of the table, corresponding 
to the $\alpha$-LMC algorithm under different sets of assumptions. 
The result for the more constrained Hessian-Lipschitz case is not 
always better than the result when only gradient Lipschitzness is 
assumed. For instance, for $W_2$, the latter is better than the 
former when $\kappa \lesssim (\kappa_2^{1.5}p^{0.5} + 
\kappa^{1.5}) \varepsilon$, which is equivalent to $ M \lesssim 
(M_2 p^{1/2} + M^{3/2}) D^{1/2}\varepsilon$. At a very high level, 
this reflects the fact that when the condition number is large, the 
Hessian-Lipschitzness does not help to get an improved result.  
Note that the same phenomenon occurs in the strongly log-concave 
case.

\subsection{The general approach based on a log-strongly-concave 
surrogate}

We have already mentioned that the strategy we adopt here is
the one described in \citep{Dalalyan14}, consisting of replacing the 
potential of the target density by  a strongly convex surrogate. 
Prior to instantiating this approach to various sampling algorithms
under various conditions and error measuring distances, we provide 
here a more formal description of it. In the remaining ${\sf dist}$ is a general 
distance on the set of all probability measures.

We will denote by $\nu_{k,\alpha}^{{\sf Alg}}$ the distribution
of the random vector obtained after performing $k$ iterations of
the algorithm ${\sf Alg}$ with the surrogate potential  
$f_\alpha (\btheta) = f(\btheta)+\alpha\|\btheta\|_2^2/2$. Our first goal
is to establish an upper bound on the distance
between the sampling distribution $\nu_{k,\alpha}^{{\sf Alg}}$ and 
the target $\pi$. The methods we analyze here depend on the step-size $h$, as
they are discretizations 
of continuous-time diffusion processes. Thus, the obtained bound will depend on $h$. 
This bound should be so that one can make it arbitrarily small  
by choosing small $\alpha$ and $h$ and a large value of $k$.   
In a second stage, the goal is to exploit the obtained error-bounds
in order to assess the order of magnitude of the computational
complexity $K$, defined in \Cref{sec:3}, as a function of $p$, 
$\varepsilon$ and the condition number $\kappa$.

To achieve this goal, we first use the triangle inequality
\begin{align}
	{\sf dist}(\nu_{k,\alpha}^{{\sf Alg}},\pi)\le 
	{\sf dist}(\nu_{k,\alpha}^{{\sf Alg}},\pi_\alpha) +  
	{\sf dist}(\pi_\alpha,\pi).
\end{align}
Then, the second term of the right hand side of the last 
displayed equation is bounded using \Cref{prop:wass_vill}. Finally, 
the distance between the sampling density $\nu_{k,\alpha}^{{\sf Alg}}$
and the surrogate $\pi_\alpha$ is bounded using the prior work 
on sampling for log-strongly-concave distributions. 
Optimizing
over $\alpha$ leads to the best bounds on precision and complexity.

\section{Prior work}\label{sec:5}

Mathematical analysis of MCMC methods defined as discretizations of 
diffusion processes is an active area of research since several decades. 
Important early references are \citep{RobertsTweedie96, RobertsRosenthal98,
RobertsStramer02,douc2004} and the references therein. Although those 
papers do cover the multidimensional case, the guarantees they provide do
not make explicit the dependence on the dimension. In a series of 
work analyzing ball walk and hit-and-run MCMCs, \cite{Lovasz2,Lovasz1} 
put forward the importance of characterizing the dependence of the
number of iterations on the dimension of the state space. 

More recently, \cite{Dalalyan14} advocated for analyzing MCMCs obtained
from continuous time diffusion processes by decomposing the error
into two terms: a non-stationarity error of the continuous-time process and
a discretization error. A large number of works applied this kind 
of approach in various settings. \citep{Bubeck18, durmus2017, Durmus2,Durmus3} 
improved the results obtained by Dalalyan and extended them in many 
directions including non-smooth potentials and variable step-sizes. 
While previous work studied the sampling error measured by the total
variation and Waserstein distances, \citep{Cheng1} proved that similar 
results hold for the Kulback-Leibler divergence. 
\citep{cheng2018underdamped,Cheng3,dalalyan_riou_2018} investigated the case of a kinetic
Langevin diffusion, showing that it leads to improved dependence on the
dimension. A promising line of related research, initiated by 
\citep{wibisono18a,Bernton18}, is to consider the sampling distributions as
a gradient flow in a space of measures. The benefits of this approach
were demonstrated in \citep{durmus2019analysis,ma2019there}. 

Motivated by applications in Statistics and Machine Learning, many 
recent papers developed theoretical guarantees for stochastic versions
of algorithms, based on noisy gradients, see \citep{Baker2019,chatterji2018theory,
dalalyan2019user,DalalyanColt,ZouXG19, raginsky17a} and the references therein. 
A related topic is non-asymptotic guarantees for the Hamiltonian Monte 
Carlo (HMC). There is a growing literature on this in recent years, see
\citep{mangoubi2017rapid, lee2018,ZongchenChen,mangoubi2019mixing} and 
the references therein. 

In all these results, the dependence of the number of iterations on the 
inverse precision is polynomial. \citep{dwivedi2018log, ChenYuansi, Mangoubi2} 
proved that one can reduce this dependence to logarithmic by using 
Metropolis adjusted versions of the algorithms.

\section{Precision and computational complexity of the LMC}\label{sec:precision_lmc}

In this section, we present non-asymptotic upper bounds 
in the (non-strongly) convex case for the suitably adapted 
LMC algorithm for Wasserstein and bounded-Lipschitz error 
measures under two sets of assumptions: Conditions 
\ref{cond:1}-\ref{cond:2} and Conditions 
\ref{cond:1}-\ref{cond:3}. To refer to these settings, 
we will call them ``Gradient-Lipschitz'' and 
``Hessian-Lipschitz'', respectively. The main goal is 
to provide a formal justification of the rates included 
in columns 2 and 3 of the table presented in \Cref{ssec:4.2}. 
To ease notation, and since there is no risk of confusion, 
we write $\mu_2$ instead $\mu_2(\pi)$.

\subsection{The Gradient-Lipschitz setting}
First we consider the Gradient-Lipschitz setting and 
give explicit conditions on the parameters $\alpha$, $h$ 
and $K$ to have a theoretical guarantee on the sampling 
error, measured in the Wasserstein distance, of the LMC 
algorithm. 

\begin{thm}\label{th-lmc-nsc}
    Suppose that the  potential function $f$ is convex and satisfies
    \Cref{cond:1}. Let $q\in [1,2]$. Then, for every 
    $\alpha \le M/20$ and $h\le 1/(M+\alpha)$, we have
    \begin{align}
        W_q(\nu_K^{\alpha\text{\sf-LMC}},\pi)
        &\le \underbrace{{\mu_2}(1-\alpha h)^{K/2}}_        
				{\substack{\text{\rm error due to the }\\[2pt] 
				\text{\rm time  finiteness}}}+ 
				\underbrace{\left({2.1hMp/\alpha}\right)^{1/2}}_
        {\text{\rm discretization error}} +  
        \underbrace{\left(C_q \alpha \mu_2^{q+2} \right)^{1/q}
				}_{\substack{\text{\rm error due to the lack}\\[2pt]
        \text{\rm of strong-convexity}}},
    \end{align}
    where $C_q$ is a dimension free constant given by \eqref{Cq}.
\end{thm}

The proof of this result is postponed to the end of this section. 
Let us consider its consequences in the cases $q=1$ and $q=2$ 
presented in the table
of \Cref{ssec:4.2}. The general strategy is to choose the value
of $\alpha$ by minimizing the sum of the discretization error and
the error caused by the lack of strong convexity. Then, the 
parameter $h$ is chosen so that the sum of the two aforementioned
errors is smaller than $99\%$ of the target precision 
$\varepsilon{\mu_2}$. Finally, the number of iterations $K$ 
is selected in such a way that the error due to the time finiteness
is also smaller than $1\%$ of the target precision. 

Implementing this strategy for $q=1$ and $q=2$, we get the 
optimized value of $\alpha$ and the corresponding value of $h$,  
\begin{center}
\begin{tabular}{c||c}
\toprule
     $q=1$ & $q=2$  \\
     \midrule
     $\displaystyle\alpha = \frac{(2.1h Mp)^{1/3}}{44^{2/3}\mu_2^2}\quad
    h = \frac{\varepsilon^3}{322 Mp}$ &
    $\displaystyle\alpha = \frac{(2.1h Mp)^{1/2}}{111^{1/2}\mu_2^2}\quad
    h = \frac{\varepsilon^4}{3900Mp}$\\
\bottomrule     
\end{tabular}    
\end{center}

These values of  $\alpha$ and $h$ satisfy the conditions imposed in 
\Cref{th-lmc-nsc}. They imply that the computational complexity of
the method, for $\nu_0 = \delta_0$ (the Dirac mass at the origin), 
is given by 
\begin{align}
    K_{\alpha\text{-LMC},W_1}(p,\varepsilon)  
    &\le \frac2{\alpha h} \log 
    \Big(\frac{100W_2(\nu_0,\pi_\alpha)}{\varepsilon{\mu_2}}\Big)
    \le 4.3\times 10^4 M\frac{\mu^2_2 p}{\varepsilon^4} 
    \log(100/\varepsilon)\\
    K_{\alpha\text{-LMC},W_2}(p,\varepsilon)  &\le \frac2{\alpha h}
    \log\Big(\frac{100W_2(\nu_0,\pi_\alpha)}{\varepsilon{\mu_2}}\Big)
    \le 3.6\times 10^6M \frac{\mu_2^2 p}{ \varepsilon^6} 
    \log(100/\varepsilon).
\end{align}
In both inequalities, the second passage is due to the 
monotone behaviour of the function  $\alpha \mapsto \mu_2 
(\pi_{\alpha})$. This property, formulated in \Cref{lem-mon-mu}, 
implies that 
\begin{align}\label{mu20}
    W_2(\delta_0,\pi_\alpha) =\mu_2(\pi_\alpha)\le \mu_2(\pi).
\end{align}
Combining \Cref{cond:2}, $M\mu_2^2(\pi) \le \kappa p^\beta$, 
with the last display, we  check that $K_{\alpha\text{-LMC},W_1} 
(p,\varepsilon)\le C \kappa (p^{1+\beta}/\varepsilon^4) 
\log(100/\varepsilon)$  and $K_{\alpha\text{-LMC},W_2} 
(p,\varepsilon)\le C \kappa (p^{1+\beta}/\varepsilon^6) \log(100/\varepsilon)$. For $\beta =1$, this matches well 
with the rates reported in the table of \Cref{ssec:4.2}. 
Unfortunately, the numerical constant $C$, 
just like the factors $4.3\times 10^4$ and $3.6\times 10^6$ in 
the last display,  is too large to be useful for practical 
purposes. Getting similar bounds with better numerical constants 
is an open question. The same remark applies to all the results
presented in the subsequent sections. 

\begin{myproof}{\Cref{th-lmc-nsc}}
    To ease notation, we write $\nu_K$ instead of $\nu_K^{\alpha
    \text{\sf -LMC}}$. The triangle inequality and the monotonicity 
    of $W_q$ with respect to $q$ imply that
    \begin{align}
        W_q(\nu_K,\pi) &\leq W_2(\nu_K,\pi_\alpha) + W_q(\pi_\alpha,\pi). 
    \end{align}
    Recall that $\pi_\alpha$ is $\alpha$-strongly log-concave 
    and has $f_\alpha$ as its potential function. By definition, $f_\alpha$  has also a Lipschitz continuous gradient with
    the Lipschitz constant at most $M+\alpha$. As we assume that 
    $h \leq {1}/{(M+\alpha)}$, we can apply
    \cite[Theorem 9]{durmus2019analysis}. It implies that
    \begin{align}
        W_2(\nu_K,\pi_\alpha) &\leq (1-\alpha h)^{K/2} W_2(\nu_0,\pi_\alpha) + \left({2h(M+\alpha)p/\alpha}\right)^{1/2}\\
        &\leq (1-\alpha h)^{K/2} W_2(\nu_0,\pi_\alpha) + \left({2.1hMp/\alpha}\right)^{1/2}.
    \end{align}
    The last inequality is true due to the fact that 
    $\alpha \leq M/20 $. Combining this inequality with  \Cref{prop:wass_vill}, we obtain
    \begin{align}
        W_q(\nu_K,\pi) & \leq (1-\alpha h)^{K/2} 
            W_2(\nu_0,\pi_\alpha) + ({2.1hMp/\alpha})^{1/2} 
            + W_q(\pi_\alpha,\pi)\\
        &\leq \mu_2(\pi_\alpha)(1-\alpha h)^{K/2}  + 
			({2.1hMp/\alpha})^{1/2} + 
			\big(C_q\alpha\mu_2^{q+2}\big)^{1/q}.
    \end{align}
    Thus, applying \eqref{mu20}, we get the claim of the 
    theorem.
\end{myproof}

\subsection{The Hessian-Lipschitz setting}

It has been noticed by \cite{Durmus2}, see also 
\cite[Theorem 5]{dalalyan2019user}, that if the potential
$f$ has a Lipschitz-continuous Hessian matrix, then the LMC
algorithm, without any modification, is more accurate than in 
the Gradient-Lipschitz setting. These improvements were 
obtained under the condition of strong convexity of the 
potential, showing that the computational complexity drops 
down from $p/\varepsilon^2$ to $p/\varepsilon$. The goal of
this section is to understand how this additional smoothness
assumption impacts the computational complexity of the
$\alpha$-LMC algorithm. 

\begin{thm}\label{thm:lmc2}
    Suppose that the  potential function $f$ satisfies 
	conditions \ref{cond:1} and \ref{cond:3}. Let $q\in[1,2]$. For every $\alpha \leq M/20$
	and $h\le 1/(M+\alpha)$, we have
	\begin{align}
        W_q(\nu_K^{\alpha\text{\sf -LMC}}\!,\pi) &\leq 
		\underbrace{\mu_2(1-\alpha h)^{K} }_
		{\substack{\text{\rm error due to the}\\[2pt] 
		\text{\rm time finiteness}}}
		+ \underbrace{\frac{M_2 h p}{2\alpha} + \frac{2.8 M^{3/2}
		hp^{1/2}}{\alpha}}_{\text{\rm discretization error}} + 
		\underbrace{\left(C_q\alpha\mu_2^{q+2}\right)}_{
		\substack{\text{\rm error due to the lack}\\[2pt]
        \text{\rm of strong-convexity}}}\!{}^{1/q}, \label{eq:lmc4-nsc}
    \end{align}
    where $C_q$ is a dimension free constant given by \eqref{Cq}.
\end{thm}
In order to provide more insight on the complexity 
bounds implied by the latter result, let us instantiate 
it for $q=1$ and $q=2$. Optimizing the sum of the two 
last error terms with respect to $\alpha$, then choosing 
this sum to be equal to $0.99\varepsilon{\mu_2}$, we 
arrive at the following values 

\begin{center}
\begin{tabular}{c||c}
\toprule
     $q=1$ & $q=2$  \\
    \midrule
    $\displaystyle\alpha = \left(\frac{hQp}{44\mu_2^{3}}
    \right)^{1/2}\quad
    h = \frac{\varepsilon^2}{ 45\mu_2 Qp}$ &
    $\displaystyle\alpha = \frac{(hQp)^{2/3}}{(111\mu_2^4)^{1/3}}
    \quad h = \frac{\varepsilon^3}{387\mu_2 Qp}$\\
\bottomrule     
\end{tabular}    
\end{center}
Here $Q$ is defined as $(M_2 + 5.6 M^{3/2}p^{-1/2})$.
These values of  $\alpha$ and $h$ satisfy the conditions imposed in 
\Cref{thm:lmc2}. They imply that the computational complexity of
the method, for $\nu_0 = \delta_0$ (the Dirac mass at the origin), 
is given by 
\begin{align}
    K_{\alpha\text{-LMC},W_1}(p,\varepsilon)  &\le 
    \frac2{\alpha h} \log\Big(\frac{100W_2(\nu_0,\pi_\alpha)}{ 
    \varepsilon{\mu_2}}\Big) \le 2\times 10^3\mu_2^3 
    Q(p/\varepsilon^3) \log(100/\varepsilon)\\
    K_{\alpha\text{-LMC},W_2}(p,\varepsilon)  &\le 
    \frac2{\alpha h} \log\Big(\frac{100W_2(\nu_0,\pi_\alpha)}{ 
    \varepsilon{\mu_2}}\Big) \le 9.9\times 10^4\mu_2^3 
    Q(p/\varepsilon^5) \log(100/\varepsilon).
\end{align}
Combining \Cref{cond:2} and the last display, we check that 
\begin{align}
K_{\alpha\text{-LMC},W_1}(p,\varepsilon)&\le 
C \varepsilon^{-3}(\kappa_2^{3/2}p^{(2+3\beta)/2}+\kappa^{3/2} p^{(1+3\beta)/2})
\log(100/\varepsilon),\\
K_{\alpha\text{-LMC},W_2}(p,\varepsilon)&\le 
C\varepsilon^{-5} (\kappa_2^{3/2}p^{(2+3\beta)/2}+\kappa^{3/2} p^{(1+3\beta)/2})
\log(100/\varepsilon).
\end{align}
The latter is true, since by definition $\kappa_2$ is equal 
to $M_2^{2/3}D$. For $\beta =1$, this matches well with the 
rates reported in the table of \Cref{ssec:4.2}.

\begin{myproof}{\Cref{thm:lmc2}}
    We repeat the same steps as in the proof of \Cref{th-lmc-nsc},
    except that instead of \citep[Theorem 9]{durmus2019analysis}
    we use \citep[Theorem 5]{dalalyan2019user}. 		
    To ease notation, we write $\nu_K$ instead of 
    $\nu_K^{\alpha\text{\sf -LMC}}$. One easily checks 
    that $\pi_\alpha$ is $\alpha$-strongly log-concave 
    with potential function $f_\alpha$. Furthermore, 
    the latter is $(M+\alpha)$-gradient-Lipschitz and
    $M_2$-Hessian-Lipschitz. Therefore, for $h \leq 
    {2}/{(M+\alpha)}$, Theorem 5 from 
    \citep{dalalyan2019user} implies that
    \begin{align}
        W_2(\nu_K,\pi_\alpha) &\leq (1-\alpha h)^{K} W_2
        (\nu_0,\pi_\alpha)  +\frac{M_2 h p}{2\alpha} + 
				\frac{13(M+\alpha)^{3/2}hp^{1/2}}{5\alpha}\\        
        &\leq (1-\alpha h)^{K} W_2(\nu_0,\pi_\alpha) + 
        \frac{M_2 h p}{2\alpha} + \frac{2.8 M^{3/2}hp^{1/2}}{\alpha},\label{eq:lmc2-nsc}
    \end{align}
    where the second inequality follows from the fact 
    that $\alpha \leq M/20 $. The triangle inequality 
    and the monotonicity of $W_q$ with respect to $q$ 
    yield $W_q(\nu_K,\pi) \leq W_2(\nu_K,\pi_\alpha) 
    + W_q(\pi_\alpha,\pi)$, which leads to 
	\begin{align}
        W_q(\nu_K,\pi) 
        &\leq {\mu_2(\pi_\alpha)}(1-\alpha h)^{K}  
            + \frac{M_2 h p}{2\alpha} + \frac{2.8 M^{3/2} 
            hp^{1/2}}{\alpha} + W_q(\pi,\pi_\alpha).
    \end{align}
	Replacing the last term above by its upper bound 
	provided by \Cref{prop:wass_vill} and applying 
	\eqref{mu20}, we get the claimed result.   
\end{myproof}

\section[Precision and computational complexity of 
KLMC and KLMC2]
{Precision and computational complexity of  KLMC 
and KLMC2}\label{sec:precision_klmc}

Several recent studies showed that for some classes 
of targets, including the strongly log-concave 
densities, the sampling error of discretizations 
of the kinetic Langevin diffusion scales better 
with the large dimension than discretizations of 
the Langevin diffusion. However, the dependence
of the available bounds on the condition number 
is better for the Langevin diffusion. In this 
section we show a similar behavior in the case of
(non-strongly) log-concave densities. This is done 
by providing quantitative upper bounds on the error 
of sampling using the kinetic Langevin process. 
 
\begin{thm}\label{thm:klmc_c1}
Suppose that the potential function $f$ satisfies 
\Cref{cond:1}. Let $q\in[1,2]$. Then for every 
$\alpha\le M/20$, $\gamma\ge \sqrt{M+2\alpha}$ and 
$h\le \alpha/(4\gamma (M+\alpha))$, we have
\begin{align}
    W_q(\nu_K^{\alpha\text{\sf-KLMC}},\pi)
    &\le \underbrace{ \sqrt{2}\,\mu_2\left(1 - 
    \frac{3\alpha h}{4\gamma}\right)^{K} 
    }_{\substack{\text{\rm error due to the time 
    finiteness}}}+ \underbrace{1.5Mp^{1/2}(h/\alpha)
    }_{\text{\rm discretization error}} +  
    \underbrace{\left(C_q \alpha \mu_2^{q+2} 
    \right)^{1/q}}_{\substack{\text{\rm error due 
    to the lack}\\[2pt]\text{\rm of strong-convexity}}}.
\end{align}
where $C_q$ is a dimension free constant given 
by \eqref{Cq}.
\end{thm}
The proof of this result is postponed to the end of 
this section. Since the contraction rate is an 
increasing function of $\gamma$, we choose its lowest 
possible value achieved for $\gamma=\sqrt{M+2\alpha}$. 
Then the strategy is the same as for the previous 
section, that is to choose the value of $\alpha$ by 
minimizing the sum of the discretization error and
the error caused by the lack of strong convexity. 
Then, the parameter $h$ is chosen so that the sum 
of the two aforementioned errors is smaller than 
$99\%$ of the target precision $\varepsilon\mu_2$. 
The number of iterations $K$ is selected in such a 
way that the error due to the time finiteness
is also smaller than $1\%$ of the target precision. 
Implementing this strategy for $q=1$ and $q=2$, we 
get the optimized value of $\alpha$ and the 
corresponding value of $h$,  
\begin{center}
\begin{tabular}{c||c}
\toprule
     $q=1$ & $q=2$  \\
     \midrule
     $\displaystyle\alpha = \frac{(1.5h Mp^{1/2})^{1/2}}{(21\mu_2^{3})^{1/2}}\quad
    h = \frac{\varepsilon^2}{143 M\mu_2p^{1/2}}$ &
    $\displaystyle\alpha = \frac{(3h Mp^{1/2})^{2/3}}{(111\mu_2^4)^{1/3}}\quad
    h = \frac{\varepsilon^4}{1200M\mu_2 p^{1/2}}$\\
\bottomrule     
\end{tabular}    
\end{center}
These values of  $\alpha$ and $h$ satisfy the 
conditions imposed in \Cref{thm:klmc_c1}. They 
imply that the computational complexity of the 
method, for $\nu_0 = \delta_0$ (the Dirac mass), 
is given by 
\begin{align}
    K_{\alpha\text{-KLMC},W_1}(p,\varepsilon)  &\le 
    \frac{4\gamma}{3\alpha h} \log\Big(\frac{150}{ 
    \varepsilon}\Big)
    \le 9.2\times 10^3(M\mu_2^2)^{3/2} (p^{1/2}/\varepsilon^3)\log(150/\varepsilon)\\
    K_{\alpha\text{-KLMC},W_2}(p,\varepsilon)  &\le 
    \frac{4\gamma}{3\alpha h} \log\Big(\frac{150}{
    \varepsilon}\Big)
    \le 4.4\times 10^5(M\mu_2^2)^{3/2} (p^{1/2}/\varepsilon^5)\log(150/\varepsilon).
\end{align}
Recall that \Cref{cond:2} implies $M\mu_2^2 \le 
\kappa p^\beta$. Combining this inequality with the 
last display, we check that 
\begin{align}
    K_{\alpha\text{-KLMC},W_q}(p,\varepsilon)&\le C 
\kappa^{3/2} (p^{(1+3\beta)/2}/\varepsilon^{2q+1}) 
\log(150/\varepsilon), \qquad q=1,2.
\end{align}
For  $\beta =1$, this matches well with the rates 
reported in the table of \Cref{ssec:4.2}.  

\begin{myproof}{\Cref{thm:klmc_c1}}
    To ease notation, we write $\nu_K$ instead of 
    $\nu_K^{\alpha\text{\sf -KLMC}}$. The triangle 
    inequality and the monotonicity of $W_q$ with 
    respect to $q$ imply that
    $W_q(\nu_K,\pi) \leq W_2(\nu_K,\pi_\alpha) + 
    W_q(\pi_\alpha,\pi)$.   
    Recall that $\pi_\alpha$ is a $\alpha$-strongly 
    log-concave distribution with potential function  
    $f_\alpha$. By definition, $f_\alpha$ has also 
    a Lipschitz continuous gradient with the Lipschitz 
    constant at most $M+\alpha$. As we assumed that 
    $\alpha\le M/20$, $\gamma\ge \sqrt{M+2\alpha}$ and 
    $h\le \alpha/(4\gamma (M+\alpha))$, we can apply
    \cite[Theorem 2]{dalalyan_riou_2018}. The latter 
    implies that
    \begin{align}
        W_2(\nu_K,\pi_\alpha) &\leq \sqrt{2}(1-3\alpha 
        h/(4\gamma))^{K} W_2(\nu_0,\pi_\alpha) + \sqrt{2}(M+\alpha)p^{1/2}(h/\alpha)\\
        &\leq \sqrt{2}\,\mu_2(\pi_\alpha)(1-3\alpha 
        h/(4\gamma))^{K}  + 1.5Mp^{1/2}(h/\alpha).
    \end{align}
    The last inequality is true thanks to the fact that 
    $\alpha \leq M/20 $. Thus, \eqref{mu20} yields
    \begin{align}
        W_q(\nu_K,\pi) & \leq \sqrt{2}\,\mu_2(\pi)( 
        1 - 3\alpha h/(4\gamma))^{K} + 1.5Mp^{1/2}
        (h/\alpha) + W_q(\pi_\alpha,\pi).
    \end{align}
    Owing to \Cref{prop:wass_vill}, the last term of the 
    last display can be bounded by $(C_q \alpha \mu_2^{q+2} 
    )^{1/q}$. This completes the proof.
\end{myproof}

The rest of this section is devoted to the sampling 
guarantees for the KLMC2 algorithm. Recall that this
algorithm requires accurate evaluations of the Hessian 
of the potential function $f$ to be available at each 
given point. 

\begin{thm}\label{thm:klmc_c2}
    Suppose that the potential function $f$ satisfies 
    conditions \ref{cond:1} and \ref{cond:3}. Let 
    $q\in[1,2]$ and $Q = M_2+M^{3/2}p^{-1/2}$. Then, 
    for every $\alpha,h,\gamma>0$ such that
    \begin{equation}
        \alpha\le\frac{M}{20},\qquad
        \gamma\ge\sqrt{M+2\alpha},\qquad 
        h\le\frac{\alpha}{5\gamma (M+\alpha)} 
        \bigvee \frac{\alpha}{4M_2\sqrt{5p}},
    \end{equation}
    we have
    \begin{align}
        W_q(\nu_K^{\alpha\text{\sf-KLMC2}},\pi)
        \le \underbrace{\sqrt{2}\,\mu_2 \left(1- 
        \frac{\alpha h}{4\gamma}\right)^K\!\!\!}_{
        \substack{\textup{error due to the}\\[2pt] 
        \textup{time finiteness}}}
        & + \underbrace{\frac{2h^2 Q p}{\alpha}
         +\frac{1.6}{\sqrt{M}}\exp\left\{-\frac{
         (\alpha/h)^2}{160M_2^2}\right\}}_{
         \textup{discretization error}}
         + \underbrace{\left(C_q \alpha \mu_2^{q+2} 
         \right)^{1/q}\!\!\!}_{\substack{\textup{ error 
         due to the lack}\\[2pt]
        \textup{ of strong-convexity}}},
    \end{align}
    where $C_q$ is a dimension free constant given 
    by \eqref{Cq}.
\end{thm}

The proof of this result is postponed to the end of 
this section. The contraction rate is an increasing 
function of $\gamma$, therefore we choose its lowest 
possible value achieved for $\gamma=\sqrt{M+2\alpha}$. 
In this case the strategy for finding  $h$ and $\alpha$ 
is slightly different from the previous ones. Here, we 
first choose the parameter $h$ so that the two terms 
of the discretization error are respectively bounded 
by $1\%$ and $2\%$ of the target precision 
$\varepsilon{\mu_2}$. This yields the following choice 
for the time step $h$:
\begin{equation}
    h=\alpha\left(160M_2^2\log\left(\frac{160}{\varepsilon
    \mu_2\sqrt{M}}\right) \bigvee\frac{100 \alpha
    Qp}{\varepsilon{\mu_2}}\right)^{-1/2}.
\end{equation}
The parameter $\alpha$ is then chosen so that the 
error due to the lack of strong convexity is lower 
than $96\%$ of the target precision. Implementing 
this strategy for $q=1$ and $q=2$, we get the 
following value for $\alpha$ 
\begin{center}
\begin{tabular}{c||c}
\toprule
    $q=1$ & $q=2$  \\
    \midrule
    $\displaystyle\alpha = \frac{\varepsilon}{23\mu_2^2}$ &
    $\displaystyle\alpha = \frac{\varepsilon^2}{116\mu_2^2}$\\
\bottomrule     
\end{tabular}    
\end{center}
Finally, the number of iterations $K$ 
is selected in such a way that the error due to the time finiteness
is also smaller than $1\%$ of the target precision. This yields, that 
\begin{equation}
    K = \frac{4\gamma}{\alpha h} \log\Big(\frac{142}{\varepsilon}\Big) 
\end{equation}
is sufficient to reach the target precision.
The values of $\gamma$, $\alpha$ and $h$ imply that the computational 
complexity of the method is given by 
\begin{align}
    K_{\alpha\text{-KLMC2},W_1}(p,\varepsilon)  &
    = 2.2\times 10^4\,\frac{M^{1/2} M_2\mu_2^4}{\varepsilon^2} \bigg\{1.6\log\Big(
		\frac{160}{\varepsilon\mu_2\sqrt{M}}\Big)\bigvee\frac{ 
		Qp}{23M_2^2\mu_2^{3}}
		\bigg\}^{1/2}\log\Big(\frac{142}{\varepsilon}\Big)\\
        K_{\alpha\text{-KLMC2},W_2}(p,\varepsilon)  &
        = 5.4\times 10^6\,\frac{M^{1/2}M_2\mu_2^4}{\varepsilon^4} 
        \bigg\{1.6\log\Big(
		\frac{160}{\varepsilon\mu_2\sqrt{M}}\Big)\bigvee\frac{
		\varepsilon Qp}{116M_2^2\mu_2^3}
		\bigg\}^{1/2}\log\Big(\frac{142}{\varepsilon}\Big).
\end{align}

Since according to \Cref{cond:2}, $\mu_2 \le D p^\beta$, 
the last display implies that up to logarithmic factors
$K_{\alpha\text{-KLMC2},W_1}
(p,\varepsilon)$ scales as $\kappa^{1/2}\kappa_2^{3/2}p^{2\beta}/
\varepsilon^2$ and $K_{\alpha\text{-KLMC2},W_2}(p,\varepsilon)$ 
scales as $\kappa^{1/2}\kappa_2^{3/2}p^{2\beta}/\varepsilon^4$. For 
$\beta =1$, this matches well with the rates reported in the table 
of \Cref{ssec:4.2}.

\begin{myproof}{\Cref{thm:klmc_c2}}
    To ease notation, we write $\nu_K$ instead of 
    $\nu_K^{\alpha\text{\sf -KLMC2}}$. As already 
    checked in the proof of \Cref{thm:lmc2}, the 
    distribution $\pi_\alpha$ is $\alpha$-strongly 
    log-concave with potential function $f_\alpha$. 
    Furthermore, the latter is $(M+\alpha)$-gradient-Lipschitz 
    and $M_2$-Hessian-Lipschitz. In view of \cite[Theorem 3]{dalalyan_riou_2018}, since the parameters $\alpha,
    \gamma,h>0$ are such that
    \begin{equation}
        \alpha\le\frac{M}{20},\qquad\gamma\ge\sqrt{M+2\alpha},
        \qquad 
        h\le\frac{\alpha}{5\gamma (M+\alpha)} \bigwedge
        \frac{\alpha}{4M_2\sqrt{5p}},
    \end{equation}
    the distribution of the KLMC2 sampler after $k$ 
    iterates satisfies
    \begin{align}
    W_2(\nu_k,\pi_\alpha)&\le\sqrt{2}\,\mu_2(\pi_\alpha)
    \left(1-\frac{\alpha h}{4\gamma}\right)^k +
    \frac{2h^2M_2p}{\alpha} + \frac{h^2(M+\alpha)^{3/2}
    \sqrt{2p}}{\alpha}\\
    &\hspace{4.25cm}
    +\frac{8h(M+\alpha)}{\alpha}\exp\left\{- 
    \frac{\alpha^2}{160M_2^2h^2}\right\}\\
    &\le \sqrt{2}\,\mu_2(\pi_\alpha)\left(1-\frac{\alpha h}{4\gamma}\right)^K + \frac{2h^2(M_2p+M^{3/2}p^{1/2})
    }{ \alpha} + \frac{1.6}{\sqrt{M}}
    \exp\left\{-\frac{(\alpha/h)^2}{160M_2^2}\right\},
    \end{align}
    where the second inequality follows from the fact that 
    $\alpha\le M/20$ and $h\le\alpha/(5\gamma (M+\alpha))$. 
    The triangle inequality and the monotonicity of $W_q$ 
    with respect to $q$ yields $W_q(\nu_K,\pi) \leq 
    W_2(\nu_K,\pi_\alpha) + W_q(\pi_\alpha,\pi)$, which 
    leads to 
	\begin{align}
        W_q(\nu_K,\pi) &\leq \sqrt{2}\,\mu_2(\pi_\alpha) \left(1-\frac{\alpha h}{4\gamma}\right)^K + 
        \frac{2h^2 Q p}{\alpha} + \frac{1.6}{\sqrt{M}}
        \exp\left\{-\frac{(\alpha/h)^2}{160M_2^2}\right\}
		+ W_q(\pi,\pi_\alpha).
    \end{align}
	Replacing the last term above by its upper bound 
	provided by \Cref{prop:wass_vill} and applying 
	\eqref{mu20}, we obtain the claim of the theorem.   
\end{myproof}

\section{Bounding the moments}\label{sec:moments}

From the user's perspective, the choice of $\alpha$ and $h$ 
requires the computation of the second moment of the 
distribution $\pi$. In most cases, this moment is an 
intractable integral. However, when some additional 
information on $\pi$ is available, this moment can
be replaced by a tractable upper bound. In this section, 
we provide upper bounds on the moments 
\begin{equation}
\mu_a^*:=\big(\bfE_{\bvartheta\sim\pi}[\|\bvartheta 
-\btheta^*\|_2^a]\big)^{1/a},\qquad a\ge 1,
\end{equation}
centered at the minimizer of the potential $\btheta^*
\in\text{argmin}_{\btheta\in\mathbb{R}^p} f(\btheta)$. 
The knowledge of the second moment is enough to 
compute the mixing times presented in 
\Cref{sec:precision_lmc} and \Cref{sec:precision_klmc}. 
However, providing bounds on general moments is of 
interest in order to highlight the dependence on the 
dimension, but also to obtain sharp numerical constants 
when the second moment is unknown. For instance, the 
proof of \Cref{prop:wass_vill} shows that results for 
the $W_1$ and $W_2$ metrics essentially rely on some 
bounds over the third and fourth moments of $\pi$, which 
could be better understood in some specific contexts.

In this section, we investigate two particular classes 
of convex functions: (a) those which are  $m$-strongly 
convex inside a ball of radius $R$ around the mode 
$\btheta^*$, and (b) those which are  $m$-strongly 
convex outside a ball of radius $R$ around the mode 
$\btheta^*$. We provide user-friendly bounds on 
$\mu_a^*$ with fairly small constants. In the 
aforementioned two cases, if $m$ and $R$ are dimension 
free, we show that $\mu_a^*$ scales respectively as 
$p\log p$ and $(p\log p)^{1/2}$. This scaling with 
the dimension is sharp, up to logarithmic factors, 
and matches \Cref{cond:2} with $\beta=2$ for the class 
(a) and \Cref{cond:2}  with $\beta=1$ for the class (b).

\begin{prop}\label{prop:moment_in}
Assume that for some positive numbers $m$ and $R$, 
we have $\nabla^2 f(\btheta)\succeq m\bfI_p$ for 
every $\btheta\in\mathbb{R}^p$ such that 
$\|\btheta-\btheta^*\|_2\le R$. Then, for every 
$a\ge 2$, we have
\begin{equation}
\mu_a^*\le  A\vee B +\frac{3}{mR\,\Gamma(p/2)^{1/a}}
\end{equation}
where\/\footnote{We denote by $\log_+ x$ the positive 
part of $\log x$, $\log_+ x = max(0,\log x)$.}
\begin{equation}
    A=\frac{3p}{mR}\left((1+a/p)\log(1 + a/p)+ \log_+ 
    \Big(\frac{2M}{m^2R^2}\Big)\right)
    \quad\text{and}\quad
    B=\left(\frac{p}{m}\right)^{1/2}\Big\{2+\frac{a}{2p}
    \Big\}^{\mathds{1}_{a>2}}.
\end{equation}
\end{prop}

In the case of fixed $a$ and $p$ tending to infinity, 
\Cref{prop:moment_in} implies that
\begin{equation}
    \mu_a^*=\tilde{O}\bigg(\frac{p}{mR}\bigvee\Big(
    \frac{p}{m}\Big)^{1/2}\bigg),
\end{equation}
where $\tilde{O}$ hides constants and polylogarithmic 
factors. In the bound of \Cref{prop:moment_in}, the term  
$A$ is the dominating one when $p/m$ is large as compared 
to $R^2$, while $B$ is the dominating term when $R^2$ is of 
a higher order of magnitude than $p/m$. The residual term $3/(mR\Gamma(p/2)^{1/a})$ goes to zero whenever $p$ 
or $R$ tend to infinity.  If $m$ and $R$ are assumed to 
be dimension free constants, then  $\mu^*_a$  scales as 
$p\log p$. This rate is optimal within a poly-log factor, 
which is proven in \Cref{lem:lower_bound_mom}. Note also 
that when $R$ goes to infinity we exactly recover the 
bound of the strongly convex case proven 
in \Cref{lem:moments_strong}.

We now switch to bounding the moments of $\pi$ under 
the condition that $f$ is convex everywhere and strongly 
convex outside the ball of radius $R$ around $\btheta^*$.

\begin{prop}\label{prop:moment_out}
Assume that for some positive $m$ and $R$, we have $\nabla^2 f(\btheta)\succeq m\bfI_p$ for every $\btheta\in\mathbb{R}^p$ 
such that $\|\btheta-\btheta^*\|_2> R$. If $p\ge 3$, then, 
for every $a>0$, we have
\begin{equation}
\mu_a^*\le \left(1 +\frac{2}{\Gamma(p/2)} \right)^{1/a}
\bigg\{(4R)\bigvee\bigg(\frac{4(p+a)}{m}\log
\Big(\frac{pM}{m}\Big)\bigg)^{1/2}\bigg\}.
\end{equation}
\end{prop}

Under the assumptions of \Cref{prop:moment_out}, we obtain
$\mu_a^*=\tilde{O}\big(R\vee({p}/{m})^{1/2}\big)$. 
In the bound of \Cref{prop:moment_out}, if $m$ and $R$ are 
assumed to be dimension free constants, then $\mu_a^*$ 
scales as $(p\log p)^{1/2}$. When $R$ is not large, this 
rate is improved in \Cref{prop:moments_out_gen} below to 
$p^{1/2}$, which is optimal. However, the bound of 
\Cref{prop:moment_out} is sharper when $R$ is large.

\begin{prop}\label{prop:moments_out_gen}
Assume that for some positive $m$ and $R$, we have $\nabla^2 f(\btheta)\succeq m\bfI_p$ for every $\btheta\in\mathbb{R}^p$ 
such that $\|\btheta-\btheta^*\|_2> R$. Then for every $a>0$ 
we have
\begin{equation}
\mu_a^* \le e^{mR^2/2a} \left(\frac{p}{m}\right)^{1/2}\left\{2 
+ \frac{a}{2p}\right\}^{\mathds{1}_{a>2}}.
\end{equation}
\end{prop}

Note that when $R$ approaches zero, this bound matches 
the one of the strongly convex case; see, for instance, \Cref{lem:moments_strong}. To close this section, let us 
note that in the setting considered in \Cref{prop:moment_out}
and \ref{prop:moments_out_gen}, one can also apply the 
bounds obtained by reflection coupling
\citep{majka2018,Cheng3}. Quite surprisingly, 
they do not lead to better bounds than those obtained in 
the present work by the simple convexification trick. 

\section{Discussion}\label{sec:disc}

In this section we highlight some consequences of the bounds
on the sampling error measured in Wasserstein distance and 
provide a discussion on how our results compare to those obtained
by other relevant approaches. 

\subsection{Moment approximation bounds derived from bounds on 
$W_q$-distances}

A glance at the table of \Cref{ssec:4.2} is enough to 
notice that the reported sampling guarantees for the $W_1$-distance 
are much better than those for $W_2$. Are there situations where 
using sampling guarantees in $W_1$ distance is more suitable than 
$W_2$ so that we can take advantage of improved rates? The answer
to this question is positive; it is formalized in the next 
proposition. In a nutshell, one can rely on guarantees in 
$W_1$-distance in the problem of approximating the expectation
of a Lipschitz function of $\bvartheta\sim \pi$, while 
guarantees in $W_2$-distance are used for approximating 
the standard-deviation of Lipschitz functions.



\begin{prop}\label{prop:sampling_vs_statistical_error}
Let $\pi$ and $\nu$ be two probability distribution on $\RR^p$; 
$\pi$ is the target distribution while $\nu$ is the sampling 
distribution. Let $\varphi:\RR^p\mapsto\RR$ be a 1-Lipschitz 
function. For $\bvartheta_1,\ldots,\bvartheta_N$ independently 
drawn from  $\nu$, define
\begin{align}
    \hat m_{N}(\varphi^q) =\frac1N\sum_{i=1}^N 
    \varphi^q(\bvartheta_i),\qquad 
    m_\pi(\varphi^q) = \bfE_{\bvartheta\sim\pi}[\varphi^q(\bvartheta)],
    \qquad
    m_\nu(\varphi^q) = \bfE_{\bvartheta\sim\nu}[\varphi^q(\bvartheta)]
\end{align}
for any $q\in\mathbb N$. It holds that
\begin{align}
    \bfE\big[\big(\hat m_{N}(\varphi) - m_{\pi}(\varphi)
    \big)^2\big]^{1/2}&\le W_1(\nu,\pi) + \sqrt{\frac{m_\nu
    (\varphi^2)}{N}},\\
    \bfE\big[\big(\hat m_{N}^{1/2}(\varphi^2) - 
    m_{\pi}^{1/2}(\varphi^2)\big)^2\big]^{1/2}&\le W_2(\nu,\pi) 
    + \sqrt{\frac{m_\nu(\varphi^4)}{N  m_\nu(\varphi^2)}}.
\end{align}
\end{prop} 
\begin{proof}
Using the bias-variance decomposition and the fact that 
$\bfE\big[\hat m_{N}(\varphi)] = m_\nu(\varphi)$, one can 
check that
\begin{align}
    \bfE\big[\big(\hat m_{N}(\varphi) - m_{\pi}(\varphi)
    \big)^2\big]& = \big( m_{\nu}(\varphi) - m_{\pi}
    (\varphi)\big)^2 + \textup{\bf Var}\big[\hat m_{N}
    (\varphi)\big]\\
    &\le W_1^2(\nu,\pi) + \frac{\textup{\bf Var}[\varphi 
    (\bvartheta_1)]}{N},
\end{align}
where the last inequality follows from the dual formulation
of the Wasserstein distance. To complete the proof of the
first claim, it suffices to note that $\textup{\bf Var}[
\varphi (\bvartheta_1)] \le m_\nu(\varphi^2)$. 

For the second claim, we start by applying the triangle 
inequality 
\begin{align}
    \bfE\big[\big(\hat m_{N}^{1/2}(\varphi^2) - 
    m_{\pi}^{1/2}(\varphi^2)\big)^2\big]^{1/2} \le 
    \bfE\big[\big(\hat m_{N}^{1/2}(\varphi^2) - 
    m_{\nu}^{1/2}(\varphi^2)\big)^2\big]^{1/2} + 
    \big|m_{\nu}^{1/2}(\varphi^2) - 
    m_{\pi}^{1/2}(\varphi^2)\big|.
\end{align}
On the one hand, for any $\bvartheta\sim \nu$ and 
$\bvartheta'\sim \pi$, we have 
\begin{align}
    \big|m_{\nu}^{1/2}(\varphi^2) - 
    m_{\pi}^{1/2}(\varphi^2)\big| &=  \big|\bfE^{1/2}[
    \varphi^2(\bvartheta)] - \bfE^{1/2}[\varphi^2( 
    \bvartheta')]\big|\\
    &\le \bfE^{1/2}\big[\big(\varphi(\bvartheta) - 
    \varphi(\bvartheta')\big)^2\big]\\
    &\le \bfE^{1/2}\big[\big\|\bvartheta - 
    \bvartheta'\big\|_2^2\big].
\end{align}
Since this is true for any coupling $(\bvartheta, 
\bvartheta')$ of $\nu$ and $\pi$, and the infimum 
of the right hand side of the last display over 
all the couplings is the $W_2$ distance, we get
\begin{align}
    \big|m_{\nu}^{1/2}(\varphi^2) - m_{\pi}^{1/2}
    (\varphi^2)\big| & \le W_2(\nu,\pi).
\end{align}
On the other hand, 
\begin{align}
    \bfE\big[\big(\hat m_{N}^{1/2}(\varphi^2) - 
    m_{\nu}^{1/2}(\varphi^2)\big)^2\big] &\le 
    \frac{\bfE\big[\big(\hat m_{N}(\varphi^2) - 
    m_{\nu}(\varphi^2) \big)^2\big]}{m_{\nu}
    (\varphi^2)}
     = \frac{\textup{\bf Var}\big[\varphi^2
     (\bvartheta_1) \big] }{Nm_{\nu}(\varphi^2)}.
\end{align}
To complete the proof of the second claim, it 
suffices to note that $\textup{\bf Var}[
\varphi^2 (\bvartheta_1)] \le m_\nu(\varphi^4)$. 
\end{proof}

The simplest application of the last result is when $\varphi 
(\btheta) = \btheta^\top\bv$ with a unit vector $\bv$. From the last 
proposition, we infer that the error of approximating the
mean of the target distribution by the sample mean of an $N$-sample 
drawn from the sampling distribution is controlled by the $W_1$ 
distance plus a small term of order $N^{-1/2}$. If, in addition
to estimating the mean, we also want to estimate the second-order
moment, then the approximation error is bounded by the $W_2$ 
distance plus a small term of order $N^{-1/2}$. Thus, the fact that
the guarantees for $W_1$ are better than those for $W_2$ illustrates
that it is computationally less expensive to approximate the 
first-order moment rather than the second-order moment. Furthermore,
the rates reported in the table of \Cref{ssec:4.2} provide a
quantitiave assessment of the computational gain.

\subsection{Suboptimality of Wasserstein bounds derived from 
total-variation bounds}

In \Cref{ssec:4.2}, we presented a summary of rates for 
the LMC with a quadratic penalty in Wasserstein distance obtained 
in this manuscript, and compared them to existing TV mixing 
times \cite{Dalalyan14, durmus2019analysis}. 
Since the topology induced by the total-variation 
distance is stronger than the topology of weak convergence 
induced by the Wasserstein distance, one can wonder whether 
existing results for TV-distance may directly lead to mixing 
times for Wasserstein distances. If the answer to this 
question is positive, the next question is how do they 
compare to the results obtained in this manuscript. The 
following proposition and the subsequent discussion aim to 
clarify this point.

\begin{prop}\label{prop:TV_imply_Wass}
Let $\nu$ and $\nu'$ be arbitrary probability distributions 
on $\mathbb{R}^p$. For every $q,r,s\ge 1$ such that $1/r+1/s=1$, 
we have
\begin{align}
    W_q(\nu,\nu') \le \left( \mu_{qr}(\nu) + \mu_{qr}(\nu')
    \right) d_{\rm TV}(\nu,\nu')^{{1}/{(qs)}}.
\end{align}
\end{prop}
\begin{proof}
The proof follows from the definition of the Wasserstein 
and the total variation distances in terms of optimal couplings. 
Let $\Gamma(\nu,\nu')$ be the set of joint distributions on $\RR^p\times\RR^p$ with marginals $\nu$ and $\nu'$ and let 
$q,r,s\ge1$ such that $1/r+1/s=1$. Choose an arbitrary coupling
$\gamma\in\Gamma(\nu,\nu')$. Applying successively 
H\"older's and Minkowski's inequalities, we arrive at
\begin{align}
    W_q^q(\nu,\nu')&\le 
    \mathbf{E}_{(X,Y)\sim\gamma}\left[\|\bvartheta - \bvartheta'
    \|_2^q \mathds{1}_{\bvartheta \neq \bvartheta'} \right]\\
    &\le \mathbf{E}_{\gamma}[\|\bvartheta - \bvartheta'
    \|_2^{qr}]^{1/r} \bfP_\gamma(\bvartheta\neq 
    \bvartheta')^{1/s}\\
    &\le \left(\mu_{qr}(\nu) + \mu_{qr}(\nu')\right)^q 
    \bfP_\gamma(\bvartheta \neq \bvartheta')^{1/s}.
\end{align}
Choosing as $\gamma$ the coupling that minimizes 
$\bfP_\gamma\left(\bvartheta\neq \bvartheta'\right)$, 
we get the claim of 
the proposition.
\end{proof}

\cref{prop:TV_imply_Wass}, combined with the available bounds
on the total variation distance, can be used to derive bounds 
on the Wasserstein error of LMCa 
or $\alpha$-LMC. In particular, for  $\nu=\pi$ and 
$\nu'=\nu_{k}^{\rm LMCa}$, one can infer from 
\cref{prop:TV_imply_Wass} that
\begin{align}
    d_{\rm TV}(\pi,\nu_{k}^{\rm LMCa})\le\left(\frac{\varepsilon\mu_2(\pi)}{\mu_{qr}(\pi) 
    +\mu_{qr}(\nu_{k}^{\rm LMCa})}\right)^{qs} \qquad 
    \Longrightarrow\qquad W_q(\pi,\nu_{k}^{\rm LMCa})\le \varepsilon\mu_2(\pi),
\end{align}
for any $q,r,s\ge 1$ such that $1/r+1/s=1$. 
As mentioned in \Cref{sec:2}, just after \Cref{cond:2}, 
there is a constant $B_{qr}$ such that $\mu_{qr}(\pi)\le 
B_{qr}\mu_2(\pi)$ for every log-concave distribution 
$\pi$. Even though the constant $B_{qr}$ does not depend 
on the target distribution, it blows up whenever 
$r\rightarrow\infty$ (known upper bounds on this constant 
are at least linear in $r$; see, for instance, 
\Cref{lemma-ck}). Therefore, assuming that there is a 
constant $C > 0$ such that $\mu_{qr}(\nu_{k}^{\rm LMCa}) 
\leq C \mu_{qr}(\pi)$, we get
\begin{align}
    d_{\rm TV}(\pi,\nu_{k}^{\rm LMCa})\le \sup_{r>1} 
    \left(\frac{\varepsilon}{(1+C) B_{qr}}\right)^{qr/(r-1)} 
    \qquad \Longrightarrow\qquad 
    W_q(\pi,\nu_{k}^{\rm LMCa}) \le \varepsilon\mu_2(\pi)    
\end{align}
for any $q\ge1$. If we neglect that the constant $B_{qr}$ blows 
up, the method sketched above leads to the best scalings with 
respect to $\varepsilon$ when $r\rightarrow +\infty$. In 
concrete examples, however, deriving from the last display 
the sharpest bound on the $W_q$-mixing-time for a fixed 
precision level $\varepsilon>0$ would require a non-trivial 
optimization with respect to $r$.

Even if we disregard the fact that $B_{qr}$ is not bounded 
and if we admit that $\mu_{qr}(\nu_{k}^{\rm LMCa}) \leq 
C \mu_{qr}(\pi)$, it turns out that the outlined method 
will lead to looser $W_q$-mixing rate than the one obtained 
by the direct approach developed in this paper. Indeed, 
choosing $r=(1+\eta)/\eta$ for some $\eta>0$ and taking 
into account that the TV-mixing-time for the LMCa obtained
in \citep{durmus2019analysis} is of the order $O(p^2/
\varepsilon^4)$, we get a $W_q$-mixing-time of order 
$O(p^2/\varepsilon^{4q(1+\eta)})$. For $q=1,2$, these 
rates are worse than the rates $p^2/\varepsilon^{2q+2}$ 
reported in the table of \Cref{ssec:4.2} for the 
$\alpha$-LMC algorithm.

The take away message is that the direct method of proof 
used in \Cref{th-lmc-nsc}-\ref{thm:klmc_c2} allows for 
a better control of $W_q$ distances than the combination 
of existing results for the TV-distance with 
\Cref{prop:TV_imply_Wass}.

\subsection{Convexification versus reflection coupling}

When the potential $f$ is strongly convex outside a ball 
of radius $R$, as in \Cref{prop:moment_out}, an alternative 
to the approach developed in the present paper is to apply
the results obtained by reflection coupling 
\citep{majka2018,Cheng3} under more general 
dissipativity assumption. We can thus compare our results
with those of these papers. 
It turns out that in the  natural setting of large $R$ our 
results provide much tighter bounds than those by reflection 
coupling. 

Indeed, let us compare \Cref{th-lmc-nsc} for 
$q=1,2$ with equations (2.28) and (2.29), Theorem 2.9, 
from \citep{majka2018}. In our results the geometric 
contraction takes place at the rate $e^{-\alpha h K/2}$, 
this rate is $e^{- c h K}$ with a parameter $c$ 
exponentially small in $R$, $c = O(e^{-aMR^2})$ for 
some $a>0$. The choice of the parameter $h$ also involves 
a term which is exponentially small in $MR^2$. The 
same factor $e^{aMR^2}$ appears in \citep{Cheng3}. 
This implies that the number of gradient evaluations to 
achieve $\varepsilon$-accuracy, derived from \citep{majka2018, 
Cheng3}, is exponential in $MR^2$. In our results, derived 
from \Cref{th-lmc-nsc} and \Cref{prop:moment_out}, this 
dependence is polynomial. In a high-dimensional setting, 
the parameter $MR^2$ will most likely be polynomial in 
$p$, leading thus to a substantial improvement obtained 
by our results as compared to those inferred from the 
reflection coupling.  

\subsection{Centering the target distribution} 

Although the theorems stated in the previous section
apply to general loc-concave distributions $\pi$ having
nonzero density on the whole $\mathbb{R}^p$, they are 
meaningful when the ``center of the distribution'' is
close to the origin. This has been quickly mentioned in
\Cref{sec:2}, just before \Cref{cond:2}; we provide below
more details on the meaning and the potential cost of 
centering. 

Clearly, if the density $\pi(\cdot)$ is log-concave with
gradient-Lipschitz potential $f$, then the same is true 
for $\pi_{\btheta_0}(\cdot) = \pi(\cdot + \btheta_0)$, 
whatever the value $\btheta_0\in \mathbb R^p$ is. The only 
quantity appearing in our upper bounds that is impacted by 
such a transformation of $\pi$ is the second-order moment
$\mu_2^2(\pi_{\theta_0}) = \bfE_{\bvartheta\sim\pi} 
[\|\bvartheta-\btheta_0\|_2^2]$. Ideally, we would like 
to choose $\btheta_0$ as the minimizer of $\mu_2(
\pi_{\theta_0})$, which corresponds to $\btheta_0 = 
\bfE_{\bvartheta\sim \pi}[\bvartheta]$.  This value,
unfortunately, is rarely available. Instead, one
can choose as $\btheta_0$ any minimizer of $f$. In view 
of \Cref{prop:moment_in} and \Cref{prop:moment_out}, 
the second-order moment of the resulting centered 
distribution has suitably bounded moments. It should be
noted, however, that computing a minimizer of the 
convex function $f$ requires $O(1/\varepsilon^2)$ 
gradient calls. Another approach that can be adopted in
a statistical setting---where $\pi$ is a posterior 
distribution---consists in choosing $\btheta_0$ as an
initial estimator of the true parameter. It can be, 
for instance, based on the method of moments.


\appendix
\section{Postponed proofs}

This section contains proofs of the propositions stated 
in previous sections as well as those of some technical 
lemmas used in the proofs of the propositions. 

\subsection{Proof of \Cref{prop:wass_vill}}

{ Without loss of generality we may assume that $\int_{\RR^p} 
\exp(-f(\btheta))\,d\btheta = 1$}. We first derive upper and 
lower bounds for  the normalizing constant of $\pi_\alpha$, 
that is
\begin{equation}
    c_{\alpha}:=\int_{\mathbb{R}^p}\pi(\btheta)\,
    e^{-\alpha\|\btheta\|_2^2/2}\,d\btheta.
\end{equation}
To do so, we introduce the notation
\begin{equation}
r_\alpha:= \frac{2}{\alpha}\log \frac{1}{c_\alpha}
\end{equation}
so that $\log(\pi_\alpha/\pi)(\btheta) = (\alpha/2)(r_\alpha - 
\|\btheta\|^2_2)$. One can check that $c_\alpha\leq 1$. 
To get a lower bound, we note that $c_\alpha$ is an 
expectation with respect to the density $\pi$, hence it 
can be lower bounded using Jensen's inequality, applied 
to the convex map $x\mapsto e^{-x}$. These two facts 
yield $\exp\{-\alpha\mu_2^2/2\} \le c_{\alpha}\le 1$. 
Therefore, by definition of $r_\alpha$, we have
\begin{equation}\label{eq:r_alpha}
    0\le r_\alpha\le\mu_2^2.
\end{equation}
For any fixed $\btheta\in\mathbb{R}^p$, we now split the 
Euclidean distance between $\pi(\btheta)$ and $\pi_\alpha 
(\btheta)$ between its positive and negative parts:
\begin{equation}
    |\pi(\btheta)-\pi_{\alpha}(\btheta)|=\underbrace{\pi(\btheta)\left[1-e^{-(\alpha/2)(\|\btheta\|_2^2-r_\alpha)}\right]\mathds{1}_{\|\btheta\|_2^2>r_\alpha}}_{:= (\pi-\pi_\alpha)_+(\btheta)}+\underbrace{\pi(\btheta)\left[e^{-(\alpha/2)(r_\alpha-\|\btheta\|_2^2)}-1\right]\mathds{1}_{\|\btheta\|_2^2<r_\alpha}}_{:=(\pi-\pi_\alpha)_-(\btheta)}.
\end{equation}
In order to bound the positive part, we make use of the inequality $1-e^{-x}\le x$ for $x>0$. Therefore:
\begin{equation}\label{eq:bound_positive_part}
    (\pi-\pi_\alpha)_+(\btheta)\le\frac{\alpha}{2}\pi(\btheta)(\|\btheta\|_2^2-r_\alpha)\mathds{1}_{\|\btheta\|_2^2>r_\alpha}.
\end{equation}
The total variation distance between densities $\pi$ and 
$\pi_\alpha$ is twice the integral of the positive part, 
$d_{\rm TV}(\pi_{\alpha},\pi) = 2\int_{\mathbb{R}^p}
(\pi-\pi_{\alpha})_+(\btheta)\,d\btheta$. Therefore, 
\begin{align}
   d_{\rm TV}(\pi_{\alpha},\pi)
    &\le\alpha\int_{\mathbb{R}^p}\pi(\btheta)
    (\|\btheta\|_2^2-r_\alpha)\mathds{1}_{\|\btheta\|_2^2 
    > r_\alpha}d\btheta \le\alpha\int_{\mathbb{R}^p} \|\btheta\|_2^2\pi(\btheta)\,d\btheta.
\end{align}
This yields the first claim of the proposition.

The proof of the bound for Wasserstein distances is inspired 
by the arguments from \citep[Theorem 6.15, page 115]{ 
villani2008optimal}. We consider a suitable coupling 
between $\pi$ and $\pi_\alpha$, defined by keeping fixed 
the mass shared by $\pi$ and $\pi_\alpha$ while 
distributing the rest of the mass with a product measure. 
Letting $C:=(\pi-\pi_\alpha)_+(\mathbb{R}^p) = 
(\pi-\pi_\alpha)_-(\mathbb{R}^p)$, we define the joint
distribution
\begin{equation}
    \gamma(d\btheta,d\btheta'):=(\pi\wedge\pi_\alpha)
    (d\btheta)\delta_{\btheta'=\btheta}+\frac{1}{C}(\pi -
    \pi_\alpha)_+(d\btheta)(\pi-\pi_\alpha)_-(d\btheta').
\end{equation}
The joint distribution $\gamma$ defines a coupling of $\pi$ and
$\pi_\alpha$. Therefore for any $q\ge 1$, by definition of the
Wasserstein distance we get
\begin{align}
W_q^q(\mu,\nu)&\le\int_{\mathbb{R}^p\times\mathbb{R}^p} 
\|\btheta-\btheta'\|_2^q\gamma(d\btheta,d\btheta')\\
&=\frac{1}{C}\int_{\mathbb{R}^p\times\mathbb{R}^p} 
\|\btheta-\btheta'\|_2^q(\pi-\pi_\alpha)_+(d\btheta) 
(\pi-\pi_\alpha)_-(d\btheta')\\
&\le \frac{1}{C}\int_{\mathbb{R}^p\times\mathbb{R}^p} 
\left(\|\btheta\|_2+\sqrt{r_{\alpha}}\right)^q 
(\pi-\pi_\alpha)_+(d\btheta)(\pi-\pi_\alpha)_-(d\btheta') \\
&=\int_{\mathbb{R}^p}\left(\|\btheta\|_2+\sqrt{r_{\alpha}}
\right)^q(\pi-\pi_\alpha)_+(d\btheta)
\end{align}
where the third line follows from the fact that 
$(\pi-\pi_\alpha)_-(d\btheta')$ has positive mass only 
inside the ball $\{\|\btheta'\|_2\le\sqrt{r_{\alpha}}\}$. 
We now define the quantity 
\begin{equation}
    J_{q,\alpha}(\pi):=\frac{1}{2}\int_{\|\btheta\|_2^2 
    > r_\alpha} \left(\|\btheta\|_2+\sqrt{r_{\alpha}}
    \right)^q \left(\|\btheta\|_2^2-r_\alpha\right)
    \pi(d\btheta)
\end{equation}
and remark that inequality \eqref{eq:bound_positive_part} 
yields
\begin{equation}
    W_q^q(\mu,\nu)\le\alpha J_{q,\alpha}(\pi).
\end{equation}
The claim of the proposition follows from the fact that there is a numerical constant $C_q$ that only depends on $q$ such that
\begin{equation}
J_{q,\alpha}(\pi)\le C_q \mu_2^{q+2}(\pi).
\end{equation}
This is a combined consequence of \eqref{eq:r_alpha} and \Cref{lemma-ck}. 
Indeed, we have
\begin{align}
    J_{q,\alpha}(\pi) &\le \frac{1}{2}\int_{\|\btheta\|_2^2 
    > r_\alpha} (2\|\btheta\|_2)^q \|\btheta\|_2^2\pi(d\btheta) 
    \le 2^{q-1} \mu_{q+2}^{q+2}(\pi) \le 2^{q-1} 
    (B_{q+2} \mu_2(\pi))^{q+2}. 
\end{align}
As shown below, we can get better values for $C_q$ 
when $q=1$ and $q=2$. We have
\begin{align}
    J_{1,\alpha}(\pi) &
    \le \frac{1}{2}\int_{\mathbb{R}^p} 
    (\|\btheta\|_2 + \mu_2) \|\btheta\|_2^2\,\pi(d\btheta)
    = (\mu_3^3+\mu_2^{3})/2\le 21 \mu_2^{3}
\end{align}
and
\begin{align}
    J_{2,\alpha}(\pi)
    & = \frac{1}{2}\int_{\|\btheta\|_2^2>r_\alpha}
    (\|\btheta\|_2+\sqrt{r_{\alpha}})^2 (\|\btheta\|_2^2-r_\alpha)\pi(d\btheta)\\
    &\le \frac{1}{2}\int_{\|\btheta\|_2^2>r_\alpha}
    (\|\btheta\|_2^4+2\|\btheta\|_2^3\sqrt{r_\alpha})
    \pi(d\btheta)\\
    &\le\frac{1}{2}\int_{\mathbb{R}^p}( \|\btheta\|_2^4 
    +2\|\btheta\|_2^3\mu_2)\pi(d\btheta)
    =(\mu_4^4+2\mu_3^3\mu_2)/2
    \le 262 \mu_2^4.
\end{align}
In both calculations, inequality \eqref{eq:r_alpha} is used to bound $r_\alpha$, while the last inequality follows from \Cref{lemma-ck}. 
It turns out that in the particular cases $q=1$ and $q=2$, 
the constant $C_q$ can be further improved using, respectively, 
\citep[Corollary 4]{Bolley05} and the transportation cost 
inequality; see for instance 
\citep[Corollary 7.2]{gozlan2010transport}. 
Since $\pi_\alpha$ is $\alpha$-strongly log-concave, we have
\begin{equation}
W_1^2(\pi,\pi_\alpha)\le 2\mu_2^2(\pi)D_{\rm KL}(\pi||\pi_\alpha),
\qquad
W_2^2(\pi,\pi_\alpha)\le(2/\alpha)D_{\rm KL}(\pi||\pi_\alpha).
\end{equation}
The computation of the Kullback-Leibler divergence yields
\begin{align}
    D_{\rm KL}(\pi||\pi_\alpha)&=\int_{\mathbb{R}^p}
    \pi(\btheta)(\alpha/2)(\|\btheta\|_2^2-r_\alpha)\,d\btheta
    =\alpha\mu^2_2/2+\log c_\alpha.
\end{align}
Using the inequality $e^{-x}\le 1-x+x^2/2$ for $x>0$ yields
\begin{equation}
   c_\alpha=\int_{\mathbb{R^p}}\pi(\btheta)e^{-(\alpha/2)
   \|\btheta\|_2^2}d\btheta\le1-\alpha\mu_2^2/2+\alpha^2\mu_4^4/8.
\end{equation}
Since $\log(1+x)\le x$ for $x>-1$ we get $D_{\rm KL} 
(\pi||\pi_\alpha) \le \alpha^2\mu^4_4/8$. Combining this 
inequality with the bound on $\mu_4$ from \Cref{lemma-ck}, 
we get
\begin{align}
    W_1 (\pi,\pi_\alpha) &\le\alpha\mu_2\mu_4^2/2\le 
        11\alpha\mu_2^3,\qquad
    W_2^2(\pi,\pi_\alpha)\le\alpha\mu^4_4/4\le 111\alpha\mu_2^4.
\end{align}
This shows that for $q=1$ and $q=2$ the constants can be improved to 
$C_1 = 11$ and $C_2=111$. Therefore we get the claim of the 
proposition.

\subsection{Proof of \Cref{prop:moment_in}}
We assume without loss of generality that $\btheta^*=\mathbf0_p$. 
Let $A\ge R$ and $a>0$.  Define $B_A=\{\btheta\in\mathbb{R}^p: 
\|\btheta\|_2\le A\}$. We split the integral into two parts
\begin{equation}
    \int_{\mathbb{R}^p}\|\btheta\|_2^a\,\pi(\btheta)\,d\btheta 
    = \int_{B_A}\|\btheta\|_2^a\,\pi(\btheta)\,d\btheta + 
        \int_{B_A^c}\|\btheta\|_2^a\,\pi(\btheta)\,d\btheta.
\end{equation}
Let us bound the integral over $B_A^c$. It follows from the 
assumptions of the proposition that for any $\btheta\in\mathbb{R}^p$, $\nabla^2f(\btheta)\succeq 
m(\|\btheta\|_2)\bfI_p$, for the map $m(r):=m\mathds{1}_{(0,R)}(r)$. 
One can show that (see \Cref{lem:moment_bound_gen} for the precise 
statement and the proof)
\begin{equation}\label{int1}
    \int_{B_A^c}\|\btheta\|_2^a\,\pi(\btheta)\,d\btheta
    \le \frac{2(M/2)^{p/2}}{\Gamma(p/2)}\int_A^{+\infty}r^{p+a-1}
    e^{-\tilde m(r)r^2/2}dr, 
\end{equation}
where $ \tilde{m}(r)\ := 2 \int_0^1 (1-t)\,m(rt)\,dt$. 
Using the fact that $m(r):=m\mathds{1}_{(0,R)}(r)$, for all 
$r>R$, we get
\begin{align}\label{eq:mtilde}
    \tilde{m}(r)
    &=2m\int_0^{1\wedge R/r}(1-t)dt
    = m\left( \frac{2R}{r}-\frac{R^2}{r^2}\right).
\end{align}
Since in \eqref{int1} the integration is with respect to 
$r\ge A\geq R$, we have 
\begin{align}
    \int_{B_A^c}\|\btheta\|_2^a\,\pi(\btheta)\,d\btheta
    &\le \frac{2(M/2)^{p/2}}{\Gamma(p/2)}e^{mR^2/2}
    \int_A^{+\infty}r^{p+a-1}e^{-mRr}dr\\
    & = \frac{2(M/2)^{p/2}}{\Gamma(p/2)} \,
    \frac{e^{mR^2/2}}{(mR)^{a+p}}
    \int_{mRA}^{+\infty} y^{p+a-1}\,e^{-y}\,dy.
\end{align}
We now use the following inequality on the incomplete 
Gamma function from \citep{natalini2000inequalities}, 
see also  \cite{borwein2009uniform}: For all $q\ge1$ 
and all $x\ge 2(q-1)$, 
\begin{align}\label{inc:gamma}
    \int_x^{+\infty}y^{q-1}e^{-y}dy\le 2x^{q-1}e^{-x}.    
\end{align}
We apply this inequality for $q=p+a$. For $A\ge2(p+a-1)
/(mR)$, we have
\begin{align}
    \int_{mRA}^{+\infty} y^{p+a-1}\,e^{-y}\,dy
    &\le  2 (mRA)^{p+a-1}e^{-mRA}.
\end{align}
Now, we make use of the fact that $mRA \ge mRA/2$. 
The latter yields
\begin{equation}
\int_{B_A^c}\|\btheta\|_2^a\,\pi(\btheta)\,d\btheta\le\frac{2^{a+1}}{(mR)^a\Gamma(p/2)}\left(\frac{2M}{m^2R^2}\right)^{p/2}\left(\frac{mRA}{2}\right)^{p+a-1}e^{-mRA/2}.
\end{equation}
The last bound ensures that the inequality
\begin{equation}\label{eq:bound_BAc}
    \int_{B_A^c}\|\btheta\|_2^a\,\pi(\btheta)\,d\btheta\le \frac{2^{a+1}}{(mR)^a\Gamma(p/2)}
\end{equation}
is fulfilled whenever $\varphi(x):= x-c\log(x)-b\ge0$, where
\begin{equation}
x=\frac{mRA}{2},\qquad c=p+a-1,\qquad b=\frac{p}{2}\log\left(\frac{2M}{m^2R^2}\right).
\end{equation}
We now establish for which values of $x$ (or equivalently, $A$) we have $\varphi(x)\ge0$. Taylor's expansion around $y_c:=1.5(c+1)\log (c+1)$ yields
\begin{equation}
\varphi\big(y_c+3b_+\big)=\varphi(y_c)
+\varphi'(y)\times 3b_+
\end{equation}
for some $y\ge y_c$. The latter implies that 
\begin{equation}
	\varphi'(y) = 1 - \frac{c}{y} \ge 1-\frac{c}{y_c}\ge 1/3.
\end{equation} 
Hence,  $\varphi\big(y_c+3b_+\big) \ge \varphi(y_c) + b_+ \ge 
y_c-c\log y_c +b_+-b\ge 0$. Since the map $\varphi$ is increasing 
on $[c,+\infty)$ and 
$y_c+3b_+\ge c$, we conclude that \eqref{eq:bound_BAc} 
is fulfilled for any
\begin{equation}
A\ge A_0:= \frac{3}{mR}\left((p+a)\log(p+a)+p \log_+\left(\frac{2M}{m^2R^2}\right)\right).
\end{equation}
We choose $A = A_0\vee R$. If  $R<A_0$,  we have $A= A_0$
and we use the obvious inequality
\begin{equation}
    \int_{B_{A_0}}\|\btheta\|_2^a\,\pi(\btheta)\,d\btheta\le A_0^a.
\end{equation}
The second case to consider is $R\ge A_0$. The map $f(\btheta)=-\log\pi(\btheta)$ being $m$-strongly convex 
on the ball $B_A=B_R$, \Cref{lem:moments_strong} yields
\begin{equation}
    \bigg(\int_{B_A}\|\btheta\|_2^a\,\pi(\btheta)\,d\btheta\bigg)^{1/a}
    \le\left(\frac{p}{m}\right)^{1/2}\left\{2 + \frac{a}{2p}
    \right\}^{\mathds{1}_{a>2}}.
\end{equation}
Since inequality \eqref{eq:bound_BAc} is valid in 
both cases, the claim of \Cref{prop:moment_in} follows.

\subsection{Proof of \Cref{prop:moment_out}}
Note that for any $\btheta\in\mathbb{R}^p$, $\nabla^2f(\btheta)\succeq m(\|\btheta\|_2)\bfI_p$, where $m(\cdot)$ is defined as below:
\begin{equation}
m(r)=m\mathds{1}_{(R,+\infty)}(r).
\end{equation}
We begin by computing the map $\tilde{m}(r):= 2\int_0^1 m(ry)(1-y)dy$. Using the definition of $\tilde m$, we have:
\begin{align}
    \tilde{m}(r) 
    &=2\int_0^1m\mathds{1}_{(R,+\infty)}(ry)(1-y)dy\\
    &=2m\mathds{1}_{r>R}\int_{R/r}^{1}(1-y)dy\\
    &=m\left(1-R/r\right)^2\mathds{1}_{r>R}.\\
\end{align}
Let $A\ge 4R$ and $a>0$. We assume without loss of generality that $\btheta^*=\mathbf0_p$. Define $B_A=\{\btheta\in\mathbb{R}^p:\|\btheta\|_2\le A\}$.
We will use the following bound:
\begin{equation}
\int_{\mathbb{R}^p}\|\btheta\|_2^a\,\pi(\btheta)\,d\btheta\le A^a+\int_{B_A^c}\|\btheta\|_2^a\,\pi(\btheta)\,d\btheta.
\end{equation}
For the second term, \Cref{lem:moment_bound_gen} yields
\begin{equation}
\int_{B_A^c}\|\btheta\|_2^a\,\pi(\btheta)\,d\btheta\le \frac{2(M/2)^{p/2}}{\Gamma(p/2)}\int_A^{+\infty}r^{p+a-1}e^{-mr^2/8}dr.
\end{equation}
This is true due to the fact that, for every  $r\ge A\ge 4R$, we have
$\tilde{m}(r)\ge m/2$. We now use 
inequality \eqref{inc:gamma} with $B = 2$, $q=(p+a)/2$ and 
$mA^2/4\ge (p+a)-1/2$:
\begin{align}
    \int_A^{+\infty}r^{p+a-1}e^{-mr^2/8}dr&=2^{-1}\left(\frac{4}{m}\right)^{(p+a)/2}\int_{mA^2/4}^{+\infty} y^{(p+a)/2-1}e^{-y}dy\\
    &\le \left(\frac{4}{m}\right)^{(p+a)/2}\left(\frac{mA^2}{4}\right)^{(p+a)/2-1}e^{-mA^2/4}\\
    &=A^a\left(\frac{4}{m}\right)^{p/2}\left(\frac{mA^2}{4}\right)^{p/2-1}e^{-mA^2/4}.
\end{align}
This yields
\begin{equation}
\int_{B_A^c}\|\btheta\|_2^a\,\pi(\btheta)\,d\btheta\le\frac{2A^a}{\Gamma(p/2)}\left(\frac{2M}{m}\right)^{p/2}\left(\frac{mA^2}{4}\right)^{p/2-1}e^{-mA^2/4}.
\end{equation}
The last bound ensures that the inequality
\begin{equation}\label{eq:bound_BAc_out}
    \int_{B_A^c}\|\btheta\|_2^a\,\pi(\btheta)\,d\btheta\le \frac{2 A^a}{\Gamma(p/2)}
\end{equation}
is fulfilled whenever $\varphi(x):= x-c\log(x)-b\ge0$, where
\begin{equation}
x=\frac{mA^2}{4},\qquad c=\frac{p}{2}-1,\qquad b=\frac{p}{2}\log\left(\frac{2M}{m}\right)>0.
\end{equation}
Taylor's expansion around $y_c:=2(c+1)\log (c+1)$ yields
\begin{equation}
\varphi(y_c+2b)=\varphi(y_c)+\varphi^{'}(y)\times 2b
\end{equation}
for some $y\ge y_c$. The latter implies that \begin{equation}\varphi^{'}(y)=1-\frac{c}{y}\ge1-\frac{c}{y_c}\ge 1/3.\end{equation} 
We get $\varphi(y_c+2b)\ge y_c-c\log(y_c)+b-b\ge 0$. Since the 
map $\varphi$ is increasing on $[c,+\infty)$ and $y_c+2b\ge c$, 
we conclude that \eqref{eq:bound_BAc_out} is fulfilled for any
\begin{align}
    A^2&\ge\frac{4}{m}\left(p\log(p/2)+p\log(2M/m)\right)
    =\frac{4p}{m}\log\left(\frac{p M}{m}\right).
\end{align}
Finally, we choose $A$ such that this inequality and the two additional assumptions: $A\ge 2R$ and $mA^2/4\ge (p+a)-1/2$ hold. If $p\ge 3$ we can choose
\begin{equation}
A=(4R)\bigvee\left(\frac{4(p+a)}{m}\log\Big(\frac{p M}{m}\Big)\right)^{1/2}.
\end{equation}
This yields the claim of \Cref{prop:moment_out}.

\subsection{Proof of \Cref{prop:moments_out_gen}}
Define $f=-\log\pi$ and for any $\btheta\in\mathbb{R}^p$:
\begin{equation}
\Bar{f}(\btheta):= f(\btheta)+\frac{m}{2}\left(\|\btheta\|_2-R\right)^2\mathds{1}_{\|\btheta\|_2\le R}.
\end{equation}
For any $\btheta\in\mathbb{R}^p$, we have $\Bar{f}(\btheta)\le f(\btheta)+mR^2/2$ , this yields
\begin{equation}
\int_{\mathbb{R}^p}\|\btheta-\btheta^*\|_2^a\,\pi(\btheta)\,d\btheta
\le e^{mR^2/2}\int_{\mathbb{R}^p}\|\btheta-\btheta^*\|_2^a
e^{-\Bar{f}(\btheta)}d\btheta.
\end{equation}
Now we define the normalising constant
\begin{equation}
 \Bar{C}:=\int_{\mathbb{R}^p}e^{-\Bar{f}(\btheta)}d\btheta
\end{equation}
and the corresponding  probability density
$\Bar{\pi}(\btheta):= e^{-\Bar{f}(\btheta)}/\Bar{C}$.
The constant $\Bar{C} \leq 1$  since $f(\btheta)\le\Bar{f}(\btheta)$ for every $\btheta\in\mathbb{R}^p$. Therefore we have 
\begin{equation}
\int_{\mathbb{R}^p}\|\btheta-\btheta^*\|_2^a \,\pi(\btheta)\,d\btheta
\le e^{mR^2/2}\int_{\mathbb{R}^p}\|\btheta-\btheta^*\|_2^a\,
\Bar{\pi}(\btheta)\,d\btheta.
 \end{equation}
By construction the density $\Bar{\pi}$ is $m$-strongly log-concave. We apply \Cref{lem:moments_strong} on this last term and get the claim of 
\Cref{prop:moments_out_gen}.

\subsection{Technical lemmas}

\begin{lem}\label{lem-mon-mu} 
	Suppose that $\pi$ has a finite fourth-order moment. Then 
	$\alpha \mapsto \mu_2(\pi_{\alpha})$ is 
	 continuously differentiable and  non-increasing, when 
	$\alpha \in  [0,+\infty)$. 
\end{lem} 

\begin{proof}
	For $k\in \NN \cup \{0\}$, define  
	\begin{equation}
		h_k(\alpha) = \int_{\RR^p} \|\btheta\|_2^k \exp \left(-
		f(\btheta) - {\alpha\|\btheta\|_2^2}/2\right)d\btheta.
	\end{equation}
	If $\pi \in \mathcal{P}_{k} (\RR^p)$ then the function $h_k$ 
	is continuous on $[0;+\infty)$. Indeed, if the sequence
	$\{\alpha_n\}_{n}$ converges $\alpha_0$, when $n\rightarrow 
	+\infty$, then the function  $\|\btheta\|_2^k \exp 
	\left(-f(\btheta) - 	(\nicefrac{1}{2}){\alpha_n \|\btheta\|_2^2 
	}\right)$ is upper-bounded by $ \|\btheta\|_2^k \exp \left(
	- f(\btheta)\right)$.  Thus in view of the dominated convergence theorem, we can interchange the limit and the integral. Since, by definition, 
	\begin{equation}
		\mu_k^k(\pi_{\alpha}) = \frac{h_k(\alpha)}{h_0(\alpha)},
	\end{equation}
	we get the continuity of $\mu_2(\pi_{\alpha})$ and $\mu_4(\pi_{\alpha})$. 
	Let us now prove that $h_k(t)$ is continuously differentiable,
	when $\pi \in \mathcal{P}_{k+2} (\RR^p) $.  The integrand 
	function in the definition of $h_k$ is a continuously 
	differentiable function with respect to $t$. In addition, its
	derivative is  continuous and is as well integrable on $\RR^p$, 
	as we supposed that $\pi$ has the $(k+2)$-th moment. Therefore, 
	the Leibniz integral rule yields the following 
	\begin{equation}
		h_k'(\alpha) = -\frac{1}{2}\int_{\RR^p} \|\btheta\|_2^{k+2} 
		\exp \left(- f(\btheta) - {\alpha \|\btheta\|_2^2}/{2} 
		\right) d\btheta =  -\frac{1}{2}h_{k+2}(t).
	\end{equation}
    The latter yields the smoothness of $h_k$. Finally, in order 
    to  prove the monotonicity of $\mu_2^2(\pi_{\alpha})$, we will
	simply compute its derivative 
	\begin{align}
		\left(\mu_2^2(\pi_{\alpha})\right)' 
		&= - \frac{1}{2h_{0}(\alpha)} h_4(\alpha) - 
		\frac{h_{0}'(\alpha)}{ h_{0} (\alpha)^2}{h_2}(\alpha)\\
		&= -\frac{1}{2}\mu_4^4(\pi_{\alpha}) +  
		\frac{h_{2}^2(\alpha)}{2h_{0}(\alpha)^2} \\
		&=  \frac{1}{2}\left( \mu_2^4(\pi_{\alpha}) - \mu_4^4(\pi_{\alpha})\right).
	\end{align}
	Since the latter is always negative, this completes the 
	proof of the lemma.
\end{proof}

\begin{lem}\label{lem:moments_strong}
Let $a>0$ and $m>0$. Assume $f=-\log\pi$ is $m$-strongly convex. 
Then
\begin{equation}
    \int_{\mathbb{R}^p}\|\btheta-\btheta^*\|_2^a\,\pi(\btheta)\,
    d\btheta\le\left(\frac{p}{m}\right)^{1/2}(2 +a/2p
    )^{\mathds{1}_{a>2}}.
\end{equation}
\end{lem}

\begin{proof}
In view of \cite{Durmus2}, 
\begin{equation}
\int_{\mathbb{R}^p}\|\btheta-\btheta^*\|_2^2\,\pi(\btheta)\,d\btheta
\le\frac{p}{m}.
\end{equation}
The monotonicity of the $\mathbb{L}_a$-norm directly yields 
the claim of the lemma for $a\le 2$.

In the case $a>2$, we use Theorem 1 from \cite{harge2004convex}. 
The result is formulated as follows.
Assume that $X\sim\mathcal{N}_p(\mu,\Sigma)$ with density 
$\varphi$ and $Y$ with density $\varphi\cdot\psi$ where $\psi$ 
is a log-concave function. Then for any convex map $g:\mathbb{R}^p\mapsto\mathbb{R}$ we have
\begin{equation}
    \mathbf{E}[g(Y-\mathbf{E}[Y])]\le\mathbf{E}[g(X-\mathbf{E}[X])].
\end{equation}
Since $f=-\log\pi$ is $m$-strongly convex, the particular 
choice $\mu=\mathbf0_p$ and $\Sigma=m\bfI_p$ yields the 
log-concavity of $\pi/\varphi$. Applied to the convex map $g:\btheta\mapsto\|\btheta\|_2^a$, the inequality of 
\cite{harge2004convex} yields
\begin{equation}
    \mathbf{E}_\pi[\|\bvartheta-\mathbf{E}_\pi[\bvartheta]\|_2^a] 
    \le\mathbf{E}[\|X\|_2^a] = \left(\frac{p}{m}\right)^{a/2}
    \frac{\Gamma((p+a)/2)}{\Gamma(p/2)(p/2)^{a/2}}
\end{equation}
using known moments of the chi-square distribution.

For any $y>0$ the map $x\mapsto x^{-y}\Gamma(x+y)/\Gamma(x)$ 
goes to $1$ when $x$ goes to infinity. For convenience, we use 
an explicit bound from \cite[Theorem 4.3]{qi2012bounds}, that is
\begin{equation}
\forall y\ge 1, \qquad x^{-y}\Gamma(x+y)/\Gamma(x)\le \left(1+y/x\right)^{y-1}.
\end{equation}
When applied to $x=p/2$ and $y=a/2>1$, this yields
\begin{equation}\label{moments_strong_1}
    \mathbf{E}_\pi[\|\bvartheta-\mathbf{E}_\pi[\bvartheta]\|_2^a]
    \le \left(\frac{p}{m}\right)^{a/2}\left(1+a/p\right)^{a/2-1}.
\end{equation}
We now bound the distance between the mean and the mode
\begin{equation}
    \label{moments_strong_2}
    \|\mathbf{E}_\pi[\bvartheta]-\btheta^*\|_2\le\mathbf{E}_\pi[
    \|\bvartheta-\btheta^*\|_2]\le(p/m)^{1/2}.
\end{equation}
Using the triangle inequality, followed by \eqref{moments_strong_1} 
and \eqref{moments_strong_2}, this yields
\begin{align}
\bigg(\int_{\mathbb{R}^p}\|\btheta-\btheta^*\|_2^a\,\pi(\btheta)\,
d\btheta\bigg)^{1/a} &\le
\big(\mathbf{E}[\|\btheta-\mathbf{E}_\pi[\btheta]\|_2^a]\big)^{1/a} 
+\|\mathbf{E}_\pi[\btheta]-\btheta^*\|_2\\
&\le (p/m)^{1/2}\left(1+a/p\right)^{1/2} + (p/m)^{1/2}.
\end{align}
Using the inequality $(1+a/p)^{1/2}\le 1+ a/(2p)$, we get the 
claim of the lemma for $a>2$.
\end{proof}

\begin{lem}\label{lem:moment_bound_gen}
    Let $f:\mathbb R^p\to\mathbb R$ be a twice differentiable 
    convex function such that $\nabla^2 f(\btheta)\preceq M 
    \mathbf I_p$ for every $\btheta\in\mathbb R$, and let $\btheta^*\in\mathbb R^p$ be a minimizer of $f$. Assume 
    that there exist a measurable map $m : [0,+\infty)
    \mapsto[0,M]$ such that  $\nabla^2f(\btheta)\succeq
    m(\|\btheta\|_2)\bfI_p$ for any $\btheta\in\mathbb{R}^p$. Let 
    $a>0$ and $A>0$.  Define the ball $B_A = \{ \btheta\in 
    \mathbb{R}^p : \|\btheta - \btheta^*\|_2\le A\}$. 
    We have
    \begin{equation}
        \int_{B_A^c}\|\btheta-\btheta^*\|_2^a\,\pi(\btheta)\,
        d\btheta \le \frac{2(M/2)^{p/2}}{\Gamma(p/2)} 
        \int_A^{+\infty} r^{p+a-1} e^{-\tilde{m}(r)\,r^2/2}dr,
    \end{equation}
    where
    \begin{equation}
        \tilde{m}(r)=2\int_0^1 (1-t)\,m(tr)\,dt.
    \end{equation}
\end{lem}

\begin{proof}
    Without loss of generality, we assume that 
    $\btheta^*=\mathbf0_p$ and $f(\mathbf0_p)=0$. Therefore, 
    the density $\pi$ is such that $\pi(\btheta)=e^{-f(\btheta})/C$
    where
    \begin{equation}
        C=\int_{\mathbb{R}^p}e^{-f(\btheta)}d\btheta 
        \ge \int_{\mathbb{R}^p}\exp\big\{-M\|\btheta\|_2^2/2 
        \big\}d\btheta
\end{equation}
by the fact that $\nabla^2 f(\btheta)\preceq M\bfI_p$ for every $\btheta\in\RR^p$.

Now, for any $r>0$ and any
$\btheta\in\mathbb{R}^p$ such that $\|\btheta\|_2 = r$, Taylor's expansion around the minimum $\mathbf0_p$ yields
\begin{align}
    f(\btheta)-f(\mathbf0_p)&=\btheta^\top\left(\int_0^1\int_0^s
    \nabla^2 f(t\btheta)\, dt\,ds\right)\btheta\\
    &\ge\|\btheta\|_2^2\int_0^1\int_0^s m(t\|\btheta\|_2)\,dt
    \,ds\\
    &= r^2\int_0^1\int_0^sm(tr)\,dt\,ds\\
    &=\frac{r^2}{2}\times \underbrace{2\int_0^1 (1-t)\, 
    m(tr)\,dt}_{=\tilde{m}(r)}.
\end{align}
We combine this fact with the lower bound on $C$ to get
\begin{align}
    \int_{B_A^c}\|\btheta\|_2^a\,\pi(\btheta)\,d\btheta&\le C^{-1}\int_{\|\btheta\|_2\ge A}\|\btheta\|_2^ae^{-f(\btheta)}d\btheta\\
    &\le \left(\int_{\mathbb{R}^p}e^{-M\|\btheta\|_2^2/2}d\btheta\right)^{-1}\int_{\|\btheta\|_2\ge A}\|\btheta\|_2^ae^{-\tilde{m}(\|\btheta\|_2)\|\btheta\|_2^2/2}d\btheta\\
    &= \left(\int_{0}^{+\infty}r^{p-1}e^{-Mr^2/2}dr\right)^{-1}\int_{A}^{+\infty}r^{a+p-1}e^{-\tilde{m}(r)r^2/2}dr\\
    &= \frac{2(M/2)^{p/2}}{\Gamma(p/2)}\int_A^{+\infty}r^{a+p-1}e^{-\tilde{m}(r)r^2/2}dr
\end{align}
where the first equality follows from a change of variables 
in polar coordinates, where the volume of the sphere cancels 
out in the ratio.
\end{proof}

\begin{lem}\label{lem:lower_bound_mom}
Assume that $\pi(\btheta)\propto e^{-f(\btheta)}$, where
\begin{equation}
f(\btheta)=0.5\|\btheta\|_2^2\mathds{1}_{\|\btheta\|_2\le 1}+\|\btheta\|_2\mathds{1}_{\|\btheta\|_2> 1}.
\end{equation}
For any $a>0$ and for any $p\ge 2\vee (a-1)$,
\begin{equation}
\mu_a^a(\pi) \ge 0.1 \Gamma(p+a)/\Gamma(p)
\ge 0.1 (p-1)^a.
\end{equation}
Consequently, under assumptions of \Cref{prop:moment_in} 
(here with $m=R=1$), the upper bound $O(p)$ on $\mu_a$ is not 
improvable.
\end{lem}
\begin{proof}
Remark first that $f(\btheta)=\varphi(\|\btheta\|_2)$ where
$\varphi(r):=0.5r^2\mathds{1}_{r\le 1}+r\mathds{1}_{r> 1}$. 
We compute explicitly the moment by a change of variable in 
polar coordinates
\begin{align}
    \int_{\mathbb{R}^p}\|\btheta\|_2^a\,\pi(\btheta)\,d\btheta&=\left(\int_0^{+\infty}r^{p-1}e^{-\varphi(r)}dr\right)^{-1}\int_0^{+\infty}r^{p+a-1}e^{-\varphi(r)}dr\\
    &=\frac{\Gamma(p+a)+\int_0^1r^{p+a-1}(e^{-r^2/2}-e^{-r})dr}{
    \Gamma(p)+\int_0^1r^{p-1}(e^{-r^2/2}-e^{-r})dr}.
\end{align}
Using the fact that $(0.2)r\le e^{-r^2/2}-e^{-r}\le r$ for $0<r<1$ 
yields
\begin{align}
    \int_{\mathbb{R}^p}\|\btheta\|_2^a\,\pi(\btheta)\,d\btheta 
    &\ge\frac{\Gamma(p+a)+0.2/(p+a+1)}{\Gamma(p)+1/(p+1)}\\
    &\ge \frac{\Gamma(p+a)+0.1/(p+1)}{\Gamma(p)+1/(p+1)}\\
    &\ge (0.1)\Gamma(p+a)/\Gamma(p)
\end{align}
where the second inequality follows from the fact that $a\le p+1$ 
by assumption, while the last inequality follows from the fact that $\Gamma(\cdot)$ is an increasing  function on $[2,+\infty)$. This proves 
the claim of the lemma.
\end{proof}



\begin{lem} \label{lemma-ck}
Let $\Gamma(k,x)$ be the upper incomplete Gamma function.
Let $k> 2$ be a real number, then $\mu_k \leq A_k^{1/k}\mu_2$
where $A_k = \min_{\lambda>2, \gamma>1} A_k(\lambda,\gamma)$ with
\begin{equation}\label{ineq-lemma-ledoux}
		A_k(\lambda,\gamma) = \frac{\sqrt{\lambda - 1}}{\lambda}
		\bigg[\frac{2\sqrt{\lambda}}{\log(\lambda - 1)}
		\bigg]^k k \Gamma\Big(k,\frac{\gamma^{1/2}\log (\lambda - 1)}{2}\Big)
		+ \frac{k(\gamma\lambda)^{k/2-1}-2}{k-2}.
\end{equation}
\end{lem}

\begin{proof}
Let $\lambda>1$ be fixed throughout the proof and define 
$\mathcal A = \{\btheta \in \RR^p : \|\btheta\|_2^2 \leq 
\lambda \mu_2^2\}$. From Markov's inequality we have
\begin{equation}
    \pi(\mathcal A) \geq  1 - \frac{\bfE_\pi[\|\bvartheta 
    \|_2^2] }{\lambda\mu_2^2} = 1 - \frac{1}{\lambda}.
\end{equation}
The set $\mathcal A$ being symmetric, Proposition 2.14 
from \citep{ledoux2001concentration} implies that
\begin{equation}
    1 - \pi(s\mathcal A) \leq \pi(\mathcal A)\left(
    \frac{1-\pi(\mathcal A)}{\pi(\mathcal A)}
    \right)^{(s+1)/2},
\end{equation}
for every real number $s$ larger than $1$. Since the 
right-hand side is a decreasing function of 
$\pi(\mathcal A)$, we obtain the following bound on 
$\pi(s\mathcal A^\complement)$:
\begin{equation}\label{sAc}
    \pi(s\mathcal A^{\complement}) \leq \frac{1}{
    \lambda (\lambda - 1)^{(s-1)/2}},\qquad \forall s\ge 1.
\end{equation}
Let us introduce the random variable $\eta$ as 
${\|\bvartheta\|_2}/\mu_2$, where $\bvartheta \sim \pi$. 
It is clear that $\eqref{ineq-lemma-ledoux}$ is equivalent 
to
\begin{equation}
    \bfE[\eta^k] \leq \frac{\sqrt{\lambda - 1}}{\lambda}
		\bigg[\frac{2\sqrt{\lambda}}{\log(\lambda - 1)}
		\bigg]^k k \Gamma\Big(k,\frac{\gamma^{1/2}\log 
		(\lambda - 1)}{2}\Big) + \frac{k(\gamma\lambda
		)^{k/2-1}-2}{k-2}.
\end{equation}
Since $\eta > 0$ almost surely,
\begin{equation}
    \bfE[\eta^k] = \int_{0}^\infty\bfP(\eta^k >u)\,du
    = k \int_{0}^\infty t^{k-1} \bfP(\eta > t)\,dt.
\end{equation}
Thus, the proof of the lemma reduces to bound the tail 
of $\eta$. The definition of $\eta$ and inequality 
\eqref{sAc} yield
\begin{equation}
    \bfP(\eta > t) = \bfP( {\|\bvartheta\|_2} 
    > t\mu_2) = \pi\left(\frac{t}{\sqrt{ \lambda }}\cdot 
    \mathcal A^{\complement}\right) \leq \frac{1}{\lambda  
    (\lambda - 1)^{(t-\sqrt{\lambda})/2\sqrt{\lambda}}},
\end{equation}
for every $t>\sqrt{\lambda}$. We choose $\gamma>1$ and apply 
this inequality to $t>\sqrt{\gamma\lambda}$. For the other 
values of $t$, that is when $t\le \sqrt{\gamma\lambda}$, we 
apply Markov's inequality to get $\bfP(\eta > t)\le 
1\wedge t^{-2}$. Combining these two bounds, we 
arrive at
\begin{align}
    \bfE[\eta^k] &\leq  k \int_{\sqrt{\gamma\lambda}}^{\infty}
		\frac{t^{k-1}}{\lambda \cdot (\lambda - 1)^{(t-\sqrt{\lambda})/2\sqrt{\lambda}}} dt
		+ \int_{0}^{\sqrt{\gamma\lambda}}kt^{k-1}(1\wedge t^{-2})dt \\
    &= k \int^{\infty}_{\sqrt{\gamma\lambda}} \frac{t^{k-1}}{\lambda \cdot
		(\lambda - 1)^{(t-\sqrt{\lambda})/2\sqrt{\lambda}}}  dt + \frac{k(\gamma\lambda)^{k/2-1}-2}{k-2}.
\end{align}
The first integral of the last sum can be calculated using the upper incomplete gamma function $\Gamma(k,z)$. Indeed,
the change of variable $z = t\log (\lambda - 1)/(2\sqrt{\lambda})$ yields
\begin{align}
    \int_{\sqrt{\gamma\lambda}}^{\infty}  \frac{t^{k-1}}{\lambda \cdot (\lambda - 1)^{(t-\sqrt{\lambda})/2\sqrt{\lambda}}}  \, dt
    & = \frac{\sqrt{\lambda - 1}}{\lambda} \int_{\sqrt{\gamma\lambda}}^\infty t^{k-1}\exp\left({-\log(\lambda - 1)\frac{t}{2\sqrt{\lambda}}}\right) dt\\
    &=\frac{\sqrt{\lambda - 1}}{\lambda} \bigg[\frac{2\sqrt{\lambda}}{\log(\lambda - 1)}\bigg]^k 
    \int_{\frac{\gamma^{1/2}\log(\lambda - 1)}{2}}^\infty z^{k-1}e^{-z}\,dz\\
    & =\frac{\sqrt{\lambda - 1}}{\lambda} \bigg[\frac{2\sqrt{\lambda}}{\log(\lambda - 1)}\bigg]^k
    \Gamma\Big(k,\frac{\gamma^{1/2}\log (\lambda - 1)}{2}\Big).
\end{align}
Finally, we obtain
\begin{align}
    \bfE[\eta^k]
    &\leq k \cdot \frac{\sqrt{\lambda - 1}}{\lambda} \bigg[
		\frac{2\sqrt{\lambda}}{\log(\lambda - 1)}\bigg]^k \Gamma
		\Big(k,\frac{\gamma^{1/2}\log (\lambda - 1)}{2}\Big) + \frac{k(\gamma\lambda)^{k/2-1}-2}{k-2}.
\end{align}
This concludes the proof.
\end{proof}

\begin{figure}
    \centering
    \includegraphics[width = \textwidth]{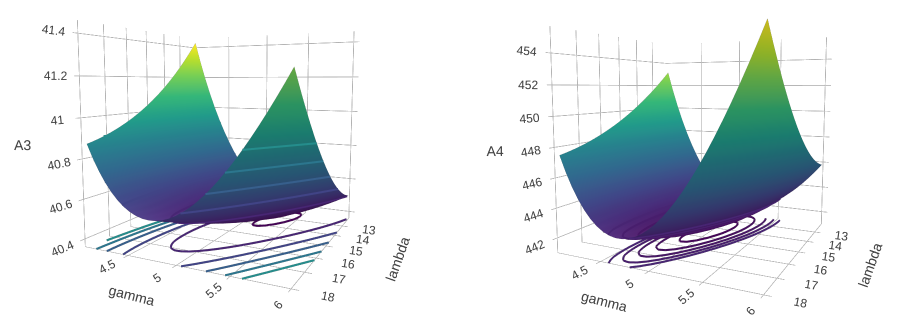}
    \caption{Shapes of the surfaces defined by the functions $A_3(\cdot,\cdot)$ 
    and $A_4(\cdot,\cdot)$, see \Cref{lemma-ck}. }
    \label{fig:constantA}
\end{figure}
\begin{remark}
We displayed\/\footnote{The R notebook for generating this 
figure can be found here\url{https://rpubs.com/adalalyan/Khintchine_constant}} in \Cref{fig:constantA} the plots of the function $A_k(\cdot,\cdot)$ 
for $k=3$ and $k=4$. Numerically, we find that the optimal choice 
for $(\lambda,\gamma)$ is approximately $\lambda=15.89$ and 
$\gamma=4.4$ for $k=3$ and $\lambda=14.97$ and $\gamma=4.8$ for 
$k=4$. This leads to the numerical bounds 
\begin{align}
A_k \le
    \begin{cases}
    40.40, & k=3,\\
    441.43, & k=4.
    \end{cases}
\end{align}
These constants are by no means optimal, but we are not aware 
of any better bound available in the literature. Inequalities 
of type $\bfE[\|\bvartheta\|_2^k] \le A_k\bfE[\|\bvartheta
\|_2^2]^{k/2}$ are often referred to as the Kintchine inequality 
\citep{Khintchine}. According to \citep{cattiaux2018poincar}, 
Corollary 4.3 from \citep{Bobkov99} implies that $A_4\le 49$ 
for one-dimensional $\bs \vartheta$ with log-concave density. 
Getting such a small constant in the multidimensional 
case would be of interest for applications to MCMC sampling. 
\end{remark}

\acks{This work was partially supported by the grant Investissements 
d’Avenir (ANR-11-IDEX-0003/Labex Ecodec/ANR-11-LABX-0047).}

{\renewcommand{\addtocontents}[2]{}
\bibliography{Literature1}}

\end{document}